\documentclass{amsart}

\usepackage{mathrsfs}
\usepackage{graphicx,amssymb,xcolor,enumerate}
\usepackage{hyperref}
\usepackage{verbatim}
\hypersetup{
    bookmarks=true,         
    pdffitwindow=false,     
    pdfstartview={FitH},    
    colorlinks=false,       
}
 \newtheorem{thm}{Theorem}[section]
 
 \newtheorem{lem}[thm]{Lemma}
 \newtheorem{prop}[thm]{Proposition}
 \theoremstyle{definition}
 \newtheorem{defn}[thm]{Definition}
 \theoremstyle{remark}
 \newtheorem{rem}[thm]{Remark}
 
 \numberwithin{equation}{section}

\usepackage{graphicx}
\usepackage{amssymb, amsmath, amsthm}
\usepackage{amsfonts}
\usepackage{amssymb}
\usepackage{float}
\usepackage{mathrsfs}
\usepackage{enumitem} 
\newcommand{\trho}{\rho_s }
\newcommand{\dis}{\displaystyle}

\newcommand{\divv}{\text{\rm div}}

\newcommand{\R}{\mathbb R}

\newcommand{\sgn}{\text{\rm sign}}

\newcommand{\intox}{\int_{\R^3}}

\newcommand{\intoxt}{\int_0^T\int_{\R^3}}
\newcommand{\intoxst}{\int_0^t\int_{\R^3}}
\newcommand{\supp}{\text{\rm supp }}
\newcommand{\intoxts}{\int_0^{t_2}\intox}

\numberwithin{equation}{section}
\numberwithin{theorem}{section}

\numberwithin{figure}{section}

\begin{document}

%
%
%
%
%
%
%
%
%








\title[Compressible flow with potential force]
 {Existence, stability and long time behaviour of weak solutions of the three-dimensional compressible Navier-Stokes equations with potential force}
 
\author{Anthony Suen} 

\address{Department of Mathematics and Information Technology\\
The Education University of Hong Kong}

\email{acksuen@eduhk.hk}

\date{\today}

\keywords{Navier-Stokes equations; compressible flow; potential force; stability; long time behaviour}

\subjclass[2000]{35Q30} 

\begin{abstract}
We address the global-in-time existence, stability and long time behaviour of weak solutions of the three-dimensional compressible Navier-Stokes equations with potential force. We show the details of the $\alpha$-dependence of different smoothing rates for weak solutions near $t=0$ under the assumption on the initial velocity $u_0$ that $u_0\in H^\alpha$ for $\alpha\in(\frac{1}{2},1]$ and obtain long time convergence of weak solutions in various norms. We then make use of the Lagrangean framework in comparing the instantaneous states of corresponding fluid particles in two different solutions. The present work provides qualitative results on the long time behaviour of weak solutions and how the weak solutions depend continuously on initial data and steady states.
\end{abstract}

\maketitle
\section{Introduction}

\subsection{Background and motivation}

We are interested in the 3-D compressible Navier-Stokes equations with an external potential force in the whole space $\R^3$ ($j=1,2,3$):
\begin{align}\label{NS} 
\left\{ \begin{array}{l}
\rho_t + \divv (\rho u) =0, \\
(\rho u^j)_t + \divv (\rho u^j u) + (P)_{x_j} = \mu\,\Delta u^j +
\lambda \, (\divv \,u)_{x_j}  + \rho f^{j}.
\end{array}\right.
\end{align}
Here $x\in\R^3$ is the spatial coordinate and $t\ge0$ stands for the time. The unknown functions $\rho=\rho(x,t)$ and $u=(u^1,u^2,u^3)(x,t)$ represent the density and velocity vector in a compressible fluid. The function $P=P(\rho)$ denotes the pressure, $f=(f^1(x),f^2(x),f^3(x))$ is a prescribed time-independent external force and $\mu$, $\lambda$ are positive viscosity constants which satisfy
\begin{align}\label{assumption on viscosity}
\mu,\lambda>0,\qquad \frac{\mu}{\lambda}>\frac{4}{5}.
\end{align}
The system \eqref{NS} is equipped with initial condition
\begin{equation}\label{IC}
(\rho(\cdot,0)-\trho,u(\cdot,0)) = (\rho_0-\trho,u_0),
\end{equation}
where the non-constant time-independent function $\rho_s =\rho_s (x)$ (known as the {\it steady state solution} to \eqref{NS}) can be obtained formally by taking $u\equiv0$ in \eqref{NS}:
\begin{align}\label{steady state}
\nabla P(\rho_s (x)) =\rho_s (x)f(x).
\end{align}

The well-posedness problem of the Navier-Stokes system \eqref{NS} is an important but challenging research topic in fluid mechanics, and we now give a brief review on the related results. The local-in-time existence of classical solution to the full Navier-Stokes equations was proved by Nash \cite{nash} and Tani \cite{tani}, and some Serrin type blow-up criteria for smooth solutions was recently obtained by Suen \cite{suen20c}. Later, Matsumura and Nishida \cite{mn1} obtained the global-in-time existence of $H^3$ solutions when the initial data was taken to be small with respect to $H^3$ norm, the results were then generalised by Danchin \cite{danchin} who showed the global existence of solutions in critical spaces. In the case of large initial data, Lions \cite{lions} obtained the existence of global-in-time finite energy weak solutions, yet the problem of uniqueness for those weak solutions remains completely open. In between the two types of solutions as mentioned above, a type of ``intermediate weak'' solutions were first suggested by Hoff in \cite{hoff95, hoff02, hoff05, hoff06} and later generalised by Li and Matsumura in \cite{LM11}, Matsumura and Yamagata in \cite{MY01}, Suen in \cite{suen13b, suen14, suen16} and other systems which include compressible magnetohydrodynamics (MHD) \cite{suenhoff12, suen12, suen20b} and compressible Navier-Stokes-Poisson system \cite{suen20a}. Solutions as obtained in this intermediate class are less regular than those small-smooth type solutions obtained by Matsumura and Nishida \cite{mn1} and Danchin \cite{danchin}, which contain possible codimension-one discontinuities in density, pressure and velocity gradient. Nevertheless, those intermediate weak solutions would be more regular than the large-weak type solutions developed by Lions \cite{lions}, hence the uniqueness and continuous dependence of solutions may be possibly obtained; see \cite{hoff06} and the compressible MHD \cite{suen20b}. 

In this present work, we address solutions of \eqref{NS} from the class of intermediate weak solutions as mentioned above. Our results give a more qualitative description on the $\alpha$-dependence of different smoothing rates for weak solutions when the initial velocity $u_0$ is in $H^\alpha$ for $\alpha\in(\frac{1}{2},1]$, which include those previous results obtained in \cite{LM11, suen13b, suen14, suen16}. In particular, we remove the smallness assumption on $\|\rho_0-\rho_s\|_{L^\infty}$ which was required in \cite{LM11} for proving global-in-time weak solutions to \eqref{NS}. Moreover, our results extend those of Cheung and Suen \cite{suen16}, which proved uniqueness of weak solution to \eqref{NS} with H\"{o}lder continuous density function $\rho$. Such condition on $\rho$ is too strong in the sense that this would exclude solutions with codimension-one singularities, which are physically interesting features for the weak solutions; see \cite{hoff02} for a detailed discussion on the propagation of singularities. The main novelties of this current work can be summarised as follows:

\noindent{\bf 1.} We strengthen the previous results obtained in \cite{LM11, suen13b, suen14, suen16}, in the way that we show the details of the $\alpha$-dependence of different smoothing rates for weak solutions near $t=0$ under the assumption that $u_0\in H^\alpha$ for $\alpha\in(\frac{1}{2},1]$; see Theorem~\ref{existence of weak sol thm}. Such regularity requirement on $u_0$ is crucial in obtaining uniqueness of the weak solutions of \eqref{NS}. It also matches with the results given in Hoff \cite{hoff02} for Navier-Stokes equations.

\noindent{\bf 2.} We provide detailed descriptions on the long time behaviour of weak solution to \eqref{NS}, including the {\it long time averages} of the ``smooth part'' of velocity $u$ and the {\it effective viscous flux}; see \eqref{long time 1}-\eqref{long time 3} in Theorem~\ref{existence of weak sol thm}. Our results are consistent with the optimal decay rates of weak solutions to \eqref{NS} which are recently obtained by Hu and Wu in \cite{HW20}.

\noindent{\bf 3.} We successfully extend the uniqueness and continuous dependence theory given in \cite{hoff06} for compressible Navier-Stokes system to \eqref{NS}, and we generalise the condition on pressure $P$ which can accommodate more general cases; see \eqref{more general condition on P} in Remark~\ref{remark on pressure}.

\noindent{\bf 4.} We remove the unnecessary H\"{o}lder continuity restriction on density from \cite{suen16} for proving uniqueness, which allows us to include a larger class of weak solutions.

\subsection{Ideas and strategies}

We give a brief description on the idea and strategies behind our work. First of all, we introduce two important functions, namely the {\it effective viscous flux} $F$ and {\it vorticity} $\omega$, which are defined by
\begin{equation}\label{def of F and omega}
\rho_s  F=(\mu+\lambda)\divv(u)-(P(\rho)-P(\rho_s )),\qquad\omega=\omega^{j,k}=u^j_{x_k}-u^k_{x_j}.
\end{equation}
By the definitions of $F$ and $\omega$, and together with \eqref{NS}$_2$, $F$ and $\omega$ satisfy the elliptic equations
\begin{align}\label{elliptic eqn for F}
\Delta (\rho_s  F)=\divv(\rho \dot{u}-\rho f+\nabla P(\rho_s )),
\end{align}
\begin{align}\label{elliptic eqn for omega}
\mu\Delta\omega=\nabla\times(\rho \dot{u}-\rho f+\nabla P(\rho_s )).
\end{align}
The functions $F$ and $\omega$ play essential roles for studying intermediate weak solutions to compressible flows, see \cite{hoff95, hoff02, suenhoff12, suen20b} for some related discussions. 

One of the key step in the derivation of {\it a priori} estimates on approximate solutions is to control the space-time integral of $|\nabla u|^4$. Due to the lack of regularity when $u_0\in L^2$, it is expected that $u(\cdot,t)\notin H^2$ and we need some other methods for controlling $\nabla u$ in $L^4$. Using the decomposition that
\begin{align*}
\Delta u^j=\divv(u_{x_j})+\omega^{j,k}_{x_k},
\end{align*}  
we have from the definition of $F$ that
\begin{align*}
(\mu+\lambda)\Delta u^j=(\rho_sF)_{x_j}+(\mu+\lambda)\omega^{j,k}_{x_k}+(P-P_s)_{x_j}.
\end{align*}
Hence in view of the above, we can obtain the desired estimate for $\nabla u$ in $L^4$ in terms of $F$, $\omega$ and $\rho-\rho_s$. In the small-time regime, we automatically have uniform-in-time estimate on $\dis\int_0^t\|(\rho-\rho_s)(\cdot,\tau)\|_{L^4}^4d\tau$ provided that $\|(\rho-\rho_s)(\cdot,\tau)\|_{L^2}<\infty$ and $\|(\rho-\rho_s)(\cdot,\tau)\|_{L^\infty}<\infty$ for all $\tau\ge0$. In the long-time regime, however, it is not obvious to obtain uniform-in-time estimate on $\dis\int_1^t\|(\rho-\rho_s)(\cdot,\tau)\|_{L^4}^4d\tau$ when $t>1$ without any smallness assumption on $\|(\rho-\rho_s)(\cdot,\tau)\|_{L^\infty}$. As inspired by the work \cite{LM11}, if we consider the auxiliary functional $\mathcal{A}(t)$ given by
\begin{align*}
\mathcal{A}(t)=\int_1^t\intox|\rho-\rho_s|^\frac{10}{3}dxd\tau,
\end{align*}
together with the mass equation \eqref{NS}$_1$, then we can show that $\mathcal{A}(t)$ is indeed bounded uniformly in time in terms of $F$ and $u$ {\it under the assumption} that $\rho$ is essentially bounded above and below. Such estimate on $\mathcal{A}(t)$ will be used for giving uniform-in-time bound on $\dis\int_1^t\|(\rho-\rho_s)(\cdot,\tau)\|_{L^4}^4d\tau$, which is important for obtaining global-in-time {\it a priori} estimates on the approximate smooth solutions to \eqref{NS}. The uniform-in-time estimate on $\mathcal{A}(t)$ is also useful for studying the long time behaviour of $\|F(\cdot,t)\|_{L^p}$ for $p\in[\frac{10}{3},6]$. The details will be carried out in Section~\ref{prelim section} and Section~\ref{proof of existence section}.

To address the uniqueness of weak solutions with respect to initial data, we need to have a better control on $\dis\int_0^t\|\nabla u(\cdot,\tau)\|_{L^\infty}d\tau$. Our attempt is to decompose $u$ as $u=u_{F}+u_{P}$, where $u_{F}$ and $u_{P}$ satisfy
\begin{align*}
\left\{
 \begin{array}{lr}
(\mu+\lambda)\Delta (u_{F})^{j}=(\rho_sF)_{x_j} +(\mu+\lambda)(\omega)^{j,k}_{x_k}\\
(\mu+\lambda)\Delta (u_P)^{j}=(P-P_s)_{x_j}.\\
\end{array}
\right.
\end{align*}
Using the {\it a priori} bounds on the effective viscous flux $F$, we can bound the integral $\dis\int_0^t\|\nabla u_F(\cdot,\tau)\|_{L^\infty}d\tau$ in terms of $F$. On the other hand, to bound the integral $\dis\int_{0}^{t}||\nabla u_{P}(\cdot,\tau)||_{L^\infty}d\tau$, we point out that $(\rho-\rho_s)\in L^2\cap L^\infty$ is {\it not} sufficient for bounding $||\nabla u_{P}(\cdot,\tau)||_{L^\infty}$. One way to overcome such difficulty is to assume that $P(\rho(\cdot,t))$ is H\"{o}lder continuous, yet such assumption would be too strong and exclude solutions with codimension-one singularities. On the other hand, we notice that if $P(\rho(\cdot,t))\in L^\infty$, then $u^j_P(\cdot,t)=(\mu+\lambda)^{-1}\Gamma_{x_j}*(P(\rho(\cdot,t))-P_s(\cdot))$ is log-Lipschitz. If we assume that the initial density is {\it piecewise H\"{o}lder continuous} (see Definition~\ref{definition of piecewise C beta}), then using the mass equation \eqref{NS}$_1$ and invoking the definition of $F$, it implies that the density is also piecewise H\"{o}lder continuous for positive time. Hence with such improved regularity on the density, it allows us to obtain the desired bound on $\dis\int_{0}^{t}||\nabla u_{P}(\cdot,\tau)||_{\infty}d\tau$. As a by-product, we further address the regularity and long-time behaviour of the ``smooth part'' $u_F$ of $u$. The details will be carried out in Section~\ref{proof of existence section}.

Once we obtain the global-in-time existence of weak solutions, we proceed to address the continuous dependence on initial data of weak solutions of the Navier-Stokes equations \eqref{NS}. As suggested by Hoff in \cite{hoff06}, weak solutions with minimal regularity are best compared in a Lagrangian framework. Instead of writing the equations \eqref{NS} in Lagrangian coordinates, we consider the {\it particle trajectories} $X(y,t,s)$ with respect to a weak solution $(\rho,u,f,\rho_s)$ to \eqref{NS}:
\begin{align*}
\left\{ \begin{array}
{lr} \dis\frac{\partial X}{\partial t}(y,t,t')
=u(X(y,t,t'),t)\\ X(y,t',t')=y.
\end{array} \right.
\end{align*}
If $(\bar\rho,\bar{u},\bar{f},\bar{\rho}_s)$ is another weak solution to \eqref{NS} with corresponding particle trajectories $\bar{X}(y,t,s)$, then we aim at estimating the difference $u-\bar{u}\circ S$, where $\bar{u}\circ S:=\bar{u}(S(x,t),t)$ and the mapping $S$ is given by
\begin{align}\label{def of S intro}
S(x,t):=\bar{X}(X(x,0,t),t,0).
\end{align}
The mapping $S$ works well with {\it convective differentiation}, in the sense that $\bar{u}\circ S$ and $\bar{u}$ enjoy the relation
\begin{align*}
(\bar{u}_t+\nabla\bar{u}\bar{u})\circ S=(\bar{u}\circ S)_t+\nabla(\bar{u}\circ S) u,
\end{align*}
hence one can obtain bounds for $u-\bar{u}\circ S$ from estimates for solutions of the adjoint of the weak equation satisfied by $u-\bar{u}\circ S$. Once we know the estimates on $u-\bar{u}\circ S$, the difference $u-\bar{u}$ can then be controlled by $u-\bar{u}\circ S$ if we have $\dis\int_0^t\|\nabla \bar{u}(\cdot,\tau)\|_{L^\infty}d\tau<\infty$, which can be achieved if the initial density $\bar{\rho}_0$ is piecewise H\"{o}lder continuous, and together with the fact that
\begin{align*}
\intox|x-S(x,t)|^2dx\le Ct\int_0^t\intox|u-\bar{u}\circ S|^2dxd\tau.
\end{align*}
The remaining details will be carried out in Section~\ref{proof of main thm section}.

It is also worth mentioning that the steady state solutions $\rho_s$ has a crucial role in addressing the stability of weak solutions to \eqref{NS} (see \eqref{bound on difference} in Theorem~\ref{Main thm}). This is not surprising: the steady state solutions $\rho_s$ mainly depend on the external force $f$ as appeared in \eqref{NS}$_2$, hence $\rho_s$ should be controlled by $f$ in some appropriate ways. In our present work, however, we choose to list out the effects of $\rho_s$ and $f$ in separate ways, which can give a clearer picture on how the stability of weak solutions depends on $\rho_s$ and $f$.

\subsection{Notations and conventions}
 
We introduce the following notations used in our work. For any $r\in(1,\infty]$ and $k\ge0$, we define the following function spaces:
\begin{align*} 
\left\{ \begin{array}{l}
L^r=L^r(\R^3), D^{k,r}=\{u\in L^1_{loc}(\R^3):\|\nabla^k u\|_{L^r}<\infty\},\|u\|_{D^{k,r}}:=\|\nabla^k u\|_{L^r}\\
W^{k,r}=L^r\cap D^{k,r}, H^k=W^{k,2}.
\end{array}\right.
\end{align*}
We adapt the following usual notations for H\"older seminorms: Given $m\ge1$, for $v:\R^3\to \R^m$ and $\alpha \in (0,1]$, 
\hfill
$$\langle v\rangle^\alpha = \sup_{{x_1,x_2\in 
\R^3}\atop{x_1\not=x_2}}
{{|v(x_2) -v(x_1)|}\over{|x_2-x_1|^\alpha}}\,;$$
and for $v:Q\subseteq\R^3 \times[0,\infty)\to \R^m$ and $\alpha_1,\alpha_2 \in (0,1]$,
\hfill
$$\langle v\rangle^{\alpha_1,\alpha_2}_{Q} = \sup_{{(x_1,t_1),(x_2,t_2)\in 
Q}\atop{(x_1,t_1)\not=(x_2,t_2)}}
{{|v(x_2,t_2) - v(x_1,t_1)|}\over{|x_2-x_1|^{\alpha_1} + |t_2-t_1|^{\alpha_2}}}\,.$$
We introduce the usual {\it convective derivative} $\dis\frac{D}{Dt}$ with respect to a velocity field $u$ as follows. For a given function $w:\R^3\times(0,T)\to\R$, we define
\begin{align}
\frac{D}{Dt}(w)=\dot{w}:=w_t+u\cdot\nabla w,
\end{align}
where $\dis w_t:=\frac{\partial w}{\partial t}$ and $\nabla w$ is the gradient of $w$. For $w:\R^3\times(0,T)\to\R^3$, we define
\begin{align}
\frac{D}{Dt}(w)=\dot{w}:=w_t+\nabla w u,
\end{align}
where $\nabla w$ is the $3 \times 3$ matrix of partial derivatives of $w$. We also write $\dis (\cdot)_{x_j}:=\frac{\partial}{\partial x_j}$ for simplicity.

We give the notion of {\it piecewise H\"{o}lder continuous} as follows (also refer to \cite{hoff02} for more details):

\begin{defn}\label{definition of piecewise C beta}
We say that a function $\phi(\cdot,\tau)$ is piecewise $C^{\beta(t)}$ if it has simple discontinuities across a $C^{1+\beta(t)}$ curve $\mathcal{C}(t):\mathcal{C}(t)=\{y(s,t):s\in I\subset\R\}$, where $\beta(t)>0$ is a function in $t$, $I$ is an open interval and the curve $\mathcal{C}(t)$ is the $u$-transport of $\mathcal C(0)$ given by:
$$y(s,t)=y(s,0)+\int_0^t u(y(s,\tau),\tau)d\tau.$$
Here $\mathcal{C}(0)$ is a $C^{\beta_0}$ curve with $\beta(0)=\beta_0>0$, which means that
\begin{equation*}
\mathcal{C}(0)=\{y_0(s):s\in\R\},
\end{equation*}
where $y(s,0)=y_0(s)$ is parameterised in arc length $s$ and $y_0$ is $C^{\beta_0}$.
\end{defn}

\begin{rem}
In view of Definition~\ref{definition of piecewise C beta}, if $\rho_0$ is piecewise continuous, then it may have simple jump discontinuities across $\mathcal C(0)$. Hence $\rho_0$ may contain codimension-one singularities in $\R^3$.
\end{rem}

\subsection{Steady state solutions}

We define $\rho_s $ as mentioned at the beginning of this section. For the pressure function $P=P(\rho)$ and the external force $f$, we assume that
\begin{align}\label{condition on P}
\mbox{$P(\rho)\in C^2((0,\infty))$ with $P(0)=0$; $P(\rho)>0$, $P'(\rho)>0$ for $\rho>0$;}
\end{align}
\begin{align}\label{condition on f}
\mbox{There exists $\Psi\in H^2$ such that $f=\nabla\Psi$ and $\Psi(x)\to 0$ as $|x|\to\infty$.}
\end{align}
Given a constant density $\rho_{\infty}>0$, we say that $(\rho_s ,0)$ is a steady state solution to \eqref{NS} if $\rho_s \in C^2(\R^3)$ and the following holds
\begin{align}
\label{eqn for steady state} \left\{ \begin{array}{l}
\nabla P(\rho_s (x)) =\rho_s (x)\nabla\Psi(x), \\
\lim\limits_{|x|\rightarrow\infty}\rho_s (x) = \rho_{\infty}.
\end{array}\right.
\end{align}
Hence $\rho_s$ satisfies the equation
\begin{align*} 
\int_{\rho_\infty}^{\rho_s}\rho^{-1} P'(\rho)d\rho=\Psi(x).
\end{align*}
To avoid vacuum state for $\rho_s$, we further assume
\begin{align}\label{condition for rho s}
-\int_0^{\rho_\infty}\frac{P'(\rho)}{\rho} d\rho<\inf_{x\in\R^3} \Psi(x)\le\sup_{x\in\R^3} \Psi(x)<\int_{\rho_\infty}^\infty \frac{P'(\rho)}{\rho} d\rho.
\end{align}
We recall the following result about the existence and uniqueness of $\rho_s$ (see for example \cite{LM11}):
\begin{prop}
Assume that $P$ and $\Psi\in H^3$ satisfy \eqref{condition on P} and \eqref{condition for rho s} respectively. Then there exists a unique solution $\rho_s$ of \eqref{eqn for steady state} satisfying $\rho_s-\rho_\infty\in H^2\cap W^{2,6}$. Moreover, there exist positive constants $\bar\rho$ and $\underline{\rho}$ depending on $\|\Psi\|_{H^3}$ such that
\begin{align}\label{bounds on rho s}
\underline{\rho}<\inf_{x\in\R^3} \rho_s(x)\le \sup_{x\in\R^3} \rho_s(x)<\bar{\rho}.
\end{align}
\end{prop}
From now on, for the sake of simplicity, we also write $P=P(\rho)$ and $P_s=P(\rho_s)$ unless otherwise specified.

\begin{rem}
It is clear that the pressure function $P$ satisfying \eqref{condition on P} includes the typical poly-tropic model $P=a\rho^\gamma$ for $a>0$ and $\gamma>0$.
\end{rem}

\subsection{Weak solutions}

Weak solutions to the system \eqref{NS} can be defined as follows. We say that $(\rho,u,f,\rho_s)$ on $\R^3\times[0,T]$ is a {\it weak solution} of \eqref{NS} if the following conditions hold:
\begin{equation}\label{condition on rho s}
\mbox{$\rho_s$ is a steady state solution to \eqref{eqn for steady state} which satisfies \eqref{condition for rho s};}
\end{equation}
\begin{equation}\label{condition on weak sol1}
\mbox{$\rho-\rho_s$ is a bounded map from $[0,T]$ into $L^1_{loc}\cap H^{-1}$ and $\rho\ge0$ a.e.;}
\end{equation}
\begin{equation}\label{condition on weak sol2}
\rho_0u_0\in L^2; \rho u,P-P_s,\nabla u,\rho f\in L^2(\R^3\times(0,T));\rho|u|^2\in L^1(\R^3\times(0,T));
\end{equation}
For all $t_2\ge t_1 \ge 0$ and $C^1$ test functions $\varphi\in\mathcal{D}(\R^3\times(-\infty,\infty))$ which are Lipschitz on $\R^3\times[t_1,t_2]$ with $\supp\varphi(\cdot,\tau)\subset K$, $\tau\in[t_1,t_2]$, where $K$ is compact and
\begin{align}\label{WF1}
\left.\int_{\R^3}\rho(x,\cdot)\varphi(x,\cdot)dx\right|_{t_1}^{t_2}=\int_{t_1}^{t_2}\int_{\R^3}(\rho\varphi_t + \rho u\cdot\nabla\varphi)dxd\tau;
\end{align}
The weak form of the momentum equation
\begin{align}\label{WF2}
\left.\int_{\R^3}(\rho u^{j})(x,\cdot)\varphi(x,\cdot)dx\right|_{t_1}^{t_2}=&\int_{t_1}^{t_2}\int_{\R^3}[\rho u^{j}\varphi_t + \rho u^{j}u\cdot\nabla\varphi + (P-P_s)\varphi_{x_j}]dxd\tau\notag\\
& - \int_{t_1}^{t_2}\int_{\R^3}[\mu\nabla u^{j}\cdot\nabla\varphi + (\mu - \xi)(\divv(u))\varphi_{x_j}]dxd\tau\\
&+ \int_{t_1}^{t_2}\int_{\R^3}(\rho f - \nabla P_s)\cdot\varphi dxd\tau.\notag
\end{align}
holds for test functions $\varphi$ which are locally Lipschitz on $\R^3 \times [0, T]$ and for which $\varphi,\varphi_t,\nabla\varphi \in L^2(\R^3 \times (0,T))$, $\nabla\varphi \in L^\infty(\R^3 \times (0,T))$, and $\varphi(\cdot,\tau) = 0$.

For the two solutions $(\rho,u,f,\rho_s)$ and $(\bar\rho,\bar{u},\bar{f},\bar{\rho}_s)$ we compare, they will be assumed to satisfy
\begin{equation}\label{spaces for weak sol 1}
u,\bar{u}\in C(\R^3\times(0,T])\cap L^1((0,T);W^{1,\infty})\cap L^\infty_{loc}((0,T];L^\infty);
\end{equation}
\begin{equation}\label{spaces for weak sol 2}
\rho-\rho_s,\bar{\rho}-\bar{\rho}_s,u,\bar{u},f,\bar{f}\in L^2(\R^2\times(0,T)).
\end{equation}
One of the solutions $(\rho,u,f,\rho_s)$ will have to satisfy
\begin{equation}\label{condition on weak sol3}
\|f\|_{L^\infty}<\infty,
\end{equation}
\begin{equation}\label{condition on weak sol4}
\rho,\rho^{-1}\in L^\infty(\R^3\times(0,T)),
\end{equation}
and 
\begin{equation}\label{condition on weak sol5}
\int_0^T\intox|u|^q dxd\tau<\infty
\end{equation}
for some $q>3$, and the other solution $(\bar\rho,\bar{u},\bar{f},\bar{\rho}_s)$ will have to satisfy
\begin{equation}\label{condition on weak sol6}
\|\bar{\rho}_s\|_{L^\infty}+\|\nabla\bar{\rho}_s\|_{L^\infty}<\infty,
\end{equation}
\begin{align}\label{condition on weak sol8}
&\int_0^T[\tau\|\nabla\bar{F}(\cdot,\tau)\|^2_{L^2}+\tau^\frac{4}{5}\|\nabla\bar{F}(\cdot,\tau)\|^{\frac{8}{5}}_{L^4}]d\tau\notag\\
&\qquad\qquad+\int_0^T[\tau\|\nabla\bar{\omega}(\cdot,\tau)\|^2_{L^2}+\tau^\frac{4}{5}\|\nabla\bar{\omega}(\cdot,\tau)\|^{\frac{8}{5}}_{L^4}]d\tau\\
&\qquad\qquad\qquad\qquad+\int_0^T[\|\bar{u}(\cdot,\tau)\|^2_{L^\infty}+\tau\|\nabla\bar{u}(\cdot,\tau)\|^2_{L^\infty}]d\tau<\infty,\notag
\end{align}
where $\bar{F}$ and $\bar{\omega}$ are as in \eqref{elliptic eqn for F}-\eqref{elliptic eqn for omega} and 
\begin{equation}\label{condition on weak sol9}
\bar{f}\in L^{2\tilde q},
\end{equation}
for some $\tilde q\in[1,\infty]$. Finally, we assume that
\begin{equation}\label{condition on weak sol10}
\rho_0-\bar{\rho}_0\in L^2\cap L^{2\tilde q'},
\end{equation}
where $\tilde q'$ is the H\"{o}lder conjugate of $\tilde q$ in \eqref{condition on weak sol9}.

\subsection{Main results}

We state the main results of this paper. First of all, the following theorem gives the global-in-time existence and long time behaviour of weak solution $(\rho,u)$ to \eqref{NS}:

\begin{thm}\label{existence of weak sol thm}
Given $\rho_{1},\rho_{2},\rho_\infty>0$, let $P$, $f$, $\lambda$, $\mu$ be the system parameters in \eqref{NS} satisfying \eqref{assumption on viscosity}, \eqref{condition on P}, \eqref{condition on f} and \eqref{condition for rho s}. The system \eqref{NS} has a global-in-time weak solution $(\rho,u)$ provided that
\begin{align}\label{smallness assumption} 
\left\{ \begin{array}{l}
\rho_{2}\le\rho_0\le\rho_{1} \text{ a.e.},\qquad\rho_0\in L^\infty,\\
\dis{\intox|u_0|^q dx<\infty},\\
C_0:=\|\rho_0-\rho_s\|^2_{L^2}+\|u_0\|^2_{H^{\alpha}}\ll1,\qquad\alpha\in(\frac{1}{2},1],
\end{array}\right.
\end{align}
where $q>6$ satisfies
\begin{align}\label{condition on q and viscosity}
\frac{\mu}{\lambda}>\frac{(q-2)^2}{4(q-1)}>\frac{4}{5},
\end{align}
and $\rho_s$ is a steady state solution satisfying \eqref{eqn for steady state} and \eqref{bounds on rho s}. The solution can be shown to satisfy conditions \eqref{condition on rho s}-\eqref{spaces for weak sol 2} and \eqref{condition on weak sol4}-\eqref{condition on weak sol5} and the energy estimates: there exist $C>0$ and $\theta>0$ such that
\begin{align}\label{energy bound on weak sol}
&\sup_{0\le \tau< \infty}\intox[\rho|u|^2+|\rho-\rho_s|^2+\sigma(\tau)^{1-\alpha}|\nabla u|^2+\sigma(\tau)^{2-\alpha}|\dot{u}|^2)dx\notag\\
&\qquad+\int_0^\infty\intox(|\nabla u|^2+\sigma(\tau)^{1-\alpha}|\dot{u}|^2+\sigma(\tau)^{2-\alpha}|\nabla\dot{u}|^2)dxd\tau\le CC_0^\theta,
\end{align}
\begin{align}\label{pointwise bound on rho weak sol}
\frac{1}{2}\rho_{2}\le\rho\le2\rho_{1} \text{ a.e.}
\end{align}
\begin{equation}\label{Holder estimates in space time}
\langle u\rangle^{\frac{1}{2},\frac{1}{4}}_{\R^3 \times [\tau,\infty)},\langle F\rangle^{\frac{1}{2},\frac{1}{4}}_{\R^3 \times [\tau,\infty)},\langle \omega\rangle^{\frac{1}{2},\frac{1}{4}}_{\R^3 \times [\tau,\infty)}\leq C(\tau)C_{0}^{\theta},\;\;\;\tau>0,
\end{equation}
where $\sigma(\tau)=\min\{1,\tau\}$, and $C(\tau)$ may depend additionally on a positive lower bound for $\tau$;
\begin{align}\label{bound on u weak sol}
\sup_{0\le \tau<\infty}\|u(\cdot,\tau)\|_{H^{\alpha}}\le CC_0^\theta,\qquad \sup_{0\le \tau<\infty}\|u(\cdot,\tau)\|_{L^{r}}\le C(r)C_0^\theta,
\end{align}
where $r\in(3,\frac{3}{1-\alpha})$ with $C(r)>0$ being a positive constant which depends only on $r$. Furthermore, if $\rho_0$ and $\rho_s$ are {\it piecewise} $C^{\beta_0}$ for some $\beta_0>0$ in the sense of Definition~\ref{definition of piecewise C beta} that
\begin{equation}\label{piecewise holder for initial rho}
\|\rho_0(\cdot)-\rho_s\|_{C^{\beta_0}_{pw}}\le N,
\end{equation}
for some $N>0$, then for each positive time $T$ and $t\in[0,T]$, there exists function $\beta(t)\in(0,\beta_0]$ and positive constant $C(N,T,C_0)$ such that $\rho$ is {\it piecewise} $C^{\beta(t)}$ on $[0,T]$ with
\begin{equation}\label{bound on time integral on nabla u}
\sup_{0\le \tau\le T}\|\rho(\cdot,\tau)-\rho_s\|_{C^{\beta(\tau)}_{pw}}+\int_0^T\|\nabla u(\cdot,\tau)\|_{L^\infty}d\tau\le C(N,T,C_0).
\end{equation}
Moreover, if we define 
\begin{align*}
u^j_P(\cdot,\tau)=(\mu+\lambda)^{-1}\Gamma_{x_j}*(P(\rho(\cdot,\tau))-P_s(\cdot)),
\end{align*}
where $P_s=P(\rho_s)$ and $\Gamma$ is the fundamental solution of the Laplace operator on $\R^3$, and
\begin{align*}
u^j_F(\cdot,\tau)=(\mu+\lambda)^{-1}[\Gamma_{x_j}*F+\mu*\omega^{j,k}],
\end{align*}
where $F$ and $\omega$ are as in \eqref{def of F and omega}, then $u=u_F+u_P$. In particular, for $\tilde\alpha\in(0,\alpha-\frac{1}{2})$, $\nu>\max\{\frac{1}{2},\alpha-\frac{1}{2}-\tilde\alpha\}$, $q_1\ge\frac{10}{3}$, $q_2\in[\frac{10}{3},6]$ and $q_3>2$, we have
\begin{align}\label{long time 0}
\lim_{T\to\infty}\frac{1}{T}\int_0^T\|(\rho-\rho_s)(\cdot,\tau)\|_{L^{q_1}}^{q_1}d\tau=0,
\end{align}
\begin{align}\label{long time 1}
\lim_{T\to\infty}\frac{1}{T^\nu}\int_0^T(\|u_F(\cdot,\tau)\|_{C^{1+\tilde\alpha}}+\|F(\cdot,\tau)\|_{C^{\tilde\alpha}})d\tau=0,
\end{align}
\begin{align}\label{long time 2}
\lim_{T\to\infty}(\|F(\cdot,T)\|_{L^{q_2}}+\|\omega(\cdot,T)\|_{L^{q_2}})=0
\end{align}
and
\begin{align}\label{long time 3}
\lim_{T\to\infty}(\|u(\cdot,T)\|_{L^{q_3}}+\|\rho-\rho_s\|_{L^\infty(\R^3\times[T,\infty))})=0.
\end{align}
\end{thm}

\begin{rem}
Using the energy estimate \eqref{energy bound on weak sol}, one can show that the {\it long time average} of $\|(\rho-\rho_s)(\cdot,\tau)\|_{L^2}^{2}$ on $[0,T]$ satisfies
\begin{align}\label{L^2 space time long time limit}
\lim_{T\to\infty}\frac{1}{T}\int_0^T\|(\rho-\rho_s)(\cdot,\tau)\|_{L^2}^{2}d\tau\le CC_0^\theta.
\end{align}
However, it is unknown whether the limit in \eqref{L^2 space time long time limit} vanishes or not.
\end{rem}

\begin{rem}
By choosing $\nu=1$ in \eqref{long time 1}, the result shows that the long time averages of $\|u_F(\cdot,\tau)\|_{C^{1+\tilde\alpha}}$ and $\|F(\cdot,\tau)\|_{C^{\tilde\alpha}}$ on $[0,T]$ vanish as $T$ goes to infinity. On the other hand, it is unknown whether there is any non-zero limit for the long time average in \eqref{long time 1} if $\nu\le\max\{\frac{1}{2},\alpha-\frac{1}{2}-\tilde\alpha\}$.
\end{rem}

Once Theorem~\ref{existence of weak sol thm} is proved, we are able to obtain the following stability result on weak solutions which is given in Theorem~\ref{Main thm}:

\begin{thm}\label{Main thm}
Given $a>0$, let $f$, $\lambda$, $\mu$ be the system parameters in \eqref{NS} satisfying \eqref{assumption on viscosity}, \eqref{condition on f} and \eqref{condition for rho s}, and assume $P$ satisfies
\begin{align}\label{extra condition on P}
P(\rho)=a\rho.
\end{align}
Assume that $(\rho_0,u_0)$ and $(\bar{\rho}_0,\bar{u}_0)$ are functions satisfying \eqref{condition on weak sol10}, \eqref{smallness assumption} and \eqref{piecewise holder for initial rho}. Then for each $T>0$ and $C>0$, there is a positive constant $M=M(T,C)$ such that if $(\rho,u,f,\rho_s)$ and $(\bar\rho,\bar{u},\bar{f},\bar{\rho}_s)$ are weak solutions of \eqref{NS} satisfying \eqref{spaces for weak sol 1}-\eqref{condition on weak sol9}, and if all the norms occurring in the above conditions are bounded by $C$, then 
\begin{align}\label{bound on difference}
&\left(\int_0^T\intox|u-\bar{u}|^2dxd\tau\right)^\frac{1}{2}+\sup_{0\le \tau\le T}\|(\rho-\bar{\rho})(\cdot,\tau)\|_{H^{-1}}\notag\\
&\qquad\le M\left[\|\rho_0-\bar{\rho}_0\|_{L^2\cap L^{2\tilde q'}}+\|\rho_0u_0-\bar{\rho}_0\bar{u}_0\|_{L^2}\right]\\
&\qquad\qquad+M\left[\left(\intox|\rho_s-\bar{\rho}_s|^2dx\right)^\frac{1}{2}+\left(\int_0^T\intox|f-\bar{f}\circ S|^2dxd\tau\right)^\frac{1}{2}\right],\notag
\end{align}
where $\tilde q'$ is given in \eqref{condition on weak sol10} and $S=S(x,t)$ is defined in \eqref{def of S intro}. 
\end{thm}
\begin{rem}\label{remark on pressure}
Similar to the case as in \cite{hoff06} and \cite{suen20b}, under a more general condition on the pressure $P$, namely
\begin{equation}\label{more general condition on P}
\sup_{0\le \tau\le T}\Big\|\nabla\Big(\frac{P(\rho(\cdot,\tau))-P(\bar{\rho}(\cdot,\tau))}{\rho(\cdot,\tau)-\bar{\rho}(\cdot,\tau)}\Big)\Big\|_{L^r}<\infty,
\end{equation}
for some $r\in[3,\infty]$, one can still obtain the same conclusion \eqref{bound on difference} from Theorem~\ref{Main thm}. Such condition on $P$ includes the one given in \cite{hoff06} and \cite{suen20b}.
\end{rem}
\begin{rem}
If we further assume that $\|\nabla\bar{f}\|_{L^\infty}<\infty$, then the term $f-\bar{f}\circ S$ can be replaced by $f-\bar{f}$ in \eqref{bound on difference}; see Remark~\ref{explanation on f} for the explanation.
\end{rem}

The rest of the paper is organised as follows. In Section~\ref{prelim section}, we give some {\it a priori} estimates on the smooth solutions to \eqref{NS}. In Section~\ref{proof of existence section}, we make use of those estimates obtained in Section~\ref{prelim section} to prove the global-in-time existence and long time behaviour of weak solution to \eqref{NS} described in Theorem~\ref{existence of weak sol thm}. Finally in Section~\ref{proof of main thm section}, we address the stability of weak solutions given in Theorem~\ref{Main thm} by making use of the Lagrangean framework and bounds on the weak solutions.

\section{A priori estimates}\label{prelim section} 

In this section we derive {\em a priori} bounds for smooth, local-in-time solutions $(\rho,u)$ of \eqref{NS}. We first derive {\it a priori} bounds for local-in-time smooth solutions under the assumption that the density $\rho$ is bounded above and below as given by \eqref{pointwise bound on rho}. We then proceed to close the estimates by obtaining the necessary bounds for density in a maximum principle argument along particle trajectories of the velocity.

Given a steady state solution  $\rho_s$ satisfying \eqref{eqn for steady state} and \eqref{bounds on rho s}, we define functionals $\Phi_{1}(t),\Phi_{2}(t),\Phi_{3}(t)$ and $\Phi(t)$ for a given solution $(\rho,u)$ by
\begin{align*}
\Phi_1(t)&=\sup_{0\le \tau\le t}(\|u(\cdot,\tau)\|_{L^2}^2+||(\rho-\rho_s)(\cdot,\tau)||_{L^2}^2) + \int_{0}^{t}\|\nabla u(\cdot,\tau)\|_{L^2}^2d\tau,\\
\Phi_2(t)&=\sup_{0\le \tau\le t}\sigma(\tau)^{1-\alpha}\|\nabla u(\cdot,\tau)\|^2_{L^2}+\int_0^t\sigma(\tau)^{1-\alpha}\|\dot{u}(\cdot,\tau)\|_{L^2}^2d\tau,\\
\Phi_3(t)&=\sup_{0\le \tau\le t}\sigma(\tau)^{2-\alpha}\|\dot{u}(\cdot,\tau)\|^2_{L^2}+\int_0^t\sigma(\tau)^{2-\alpha}\|\nabla\dot{u}(\cdot,\tau)\|^2_{L^2}d\tau,
\end{align*}
\begin{align*}
\Phi(t)=\Phi_1(t)+\Phi_2(t)+\Phi_3(t),
\end{align*}
where we recall that $\alpha\in(\frac{1}{2},1]$ and $\sigma(\tau)=\min\{1,\tau\}$. We will obtain {\it a priori} bounds on the above functionals and the results can be summarised as follows:

\begin{thm}\label{a priori bounds}
Given $\rho_{1},\rho_{2},\rho_\infty>0$ and $\gamma\ge1$, let $P$, $f$, $\lambda$, $\mu$ be the system parameters in \eqref{NS} satisfying \eqref{assumption on viscosity}, \eqref{condition on P}, \eqref{condition on f} and \eqref{condition for rho s}. Given $q>6$, there are positive constants $M, \theta$ depending on the parameters and assumptions in \eqref{assumption on viscosity}, \eqref{condition on P}, \eqref{condition on f} and \eqref{condition for rho s}, such that: if $\rho_s$ is a steady state solution satisfying \eqref{eqn for steady state} and \eqref{bounds on rho s}, and if $(\rho,u)$ is a solution of \eqref{NS} on $\R^3\times[0,T]$ with initial data $(\rho_0,u_0)\in H^3(\R^3)$ satisfying \eqref{smallness assumption} with the smallness assumption $C_0\ll1$, then we have
\begin{align}\label{pointwise bound on rho}
\mbox{$\frac{1}{2}\rho_{2}\le\rho(x,t)\le2\rho_{1}$ on $\R^3\times[0,T]$},
\end{align}
and 
\begin{align}\label{bound on Phi and H}
\mbox{$\Phi(t)\le MC_{0}^{\theta}$ on $\R^3\times[0,T]$,}
\end{align}
where $\alpha\in(\frac{1}{2},1]$ and $C_0:=\|\rho_0-\rho_s\|^2_{L^2}+\|u_0\|^2_{H^{\alpha}}$.
\end{thm}

Theorem~\ref{a priori bounds} will be proved in a sequence of lemmas. We first establish the bound \eqref{bound on Phi and H} {\it under the assumption} that the pointwise bounds in \eqref{pointwise bound on rho} hold for the density $\rho$. In order to control the space-time integral of $\rho-\rho_s$, we introduce different methods in {\it small-time} and {\it large-time} regimes, which will be given in Subsection~\ref{small time estimates} and Subsection~\ref{large time estimates} respectively. In Subsection~\ref{Pointwise bound on rho subsection}, we then close the estimates of Theorem~\ref{a priori bounds} by deriving pointwise bounds \eqref{pointwise bound on rho} for $\rho$ under the smallness assumption on $C_0$. This gives an noncontingent estimate for $(\rho,u)$ and thereby proving Theorem~\ref{a priori bounds}.

Unless otherwise specified, throughout this paper, $C$ will denote a generic positive constant which depends on the same quantities as the constant $M$ in the statement of Theorem~\ref{a priori bounds} but independent of time $t$ and the regularity of initial data. And for simplicity, we write $P=P(\rho)$ and $P_s=P(\rho_s)$, etc., without further referring.

We will make repeated use of the following Gagliardo-Nirenberg type inequalities and Sobolev imbeddings, the proof can be found in \cite{ziemer}:

\begin{prop}
For $p\in[2,6]$ and $r\in(3,\infty)$, there exists some generic constants $C>0$ and $C(r)>0$ such that for any $g\in H^1$ and $h\in W^{1,r}$, we have
\begin{align}
\|g\|^p_{L^p}&\le C\|g\|^\frac{6-p}{2}_{L^2}\|\nabla g\|^\frac{3p-6}{2}_{L^2},\label{GN1}\\
\langle h\rangle^\alpha&\le C(r)\|\nabla h\|_{L^r(\R^3)},\label{holder bound general}
\end{align}
where $\alpha=1-\frac{3}{r}$.
\end{prop}

We begin with the following energy balance law which gives bounds on $\Phi_1(t)$ for all $t\in[0,T]$.

\begin{lem}\label{L2 estimate lem}
Assume that the hypotheses and notations of Theorem~\ref{a priori bounds} are in force. Suppose that $\rho$ satisfies the pointwise bounds \eqref{pointwise bound on rho}. Then for any $t\in[0,T]$,
\begin{align}\label{L2 energy bound}
\Phi_1(t)&\le CC_0.
\end{align}
\end{lem}
\begin{proof}
Using $\nabla\Psi=-\rho_s^{-1}\nabla P_s$ on the momentum equation \eqref{NS}$_2$, we get
\begin{equation}\label{momentum eqn rewritten}
\rho\dot u+\rho(\rho^{-1}\nabla P-\rho_s^{-1}\nabla  P_s)-\mu\Delta u-\lambda\nabla(\divv(u))=0.
\end{equation}
Multiply \eqref{momentum eqn rewritten} by $u$ and integrate to obtain that for $t\in[0,T]$,
\begin{align}\label{energy balance law}
\left.\int_{\R^3}{\textstyle\frac{1}{2}}\rho|u|^{2}dx\right |_{0}^{t}dx + \int_{0}^{t}\int_{\R^3} &\rho u(\rho^{-1}\nabla P-\rho_s^{-1}\nabla  P_s)dxd\tau
\notag\\ +& \int_{0}^{t}\int_{\R^3}\left[\mu|\nabla u|^{2} + \lambda(\divv(u))^{2}\right]dxd\tau=0,
\end{align}
where the divergence of a matrix is taken row-wise.
Next we define 
\begin{align*}
G(\rho)=\int_{\rho_s}^{\rho}\int_{\rho_s}^{r}\tau^{-1}P'(\tau)d\tau dr,
\end{align*}
then since $\rho$ satisfies the bounds in \eqref{pointwise bound on rho}, we have
\begin{align*}
C^{-1}|\rho-\rho_s|^2\le G(\rho)\le C|\rho-\rho_s|^2.
\end{align*}
Using the mass equation, the second term on the left side of \eqref{energy balance law} can be written as follows
\begin{align*}
\int_{0}^{t}\int_{\R^3}&\rho u(\rho^{-1}\nabla P-\rho_s^{-1}\nabla  P_s)dxd\tau\\
&=\int_0^t\int_{\R^3}\rho u\cdot\nabla\left(\int_{\rho_s}^{\rho}r^{-1}P'(r)dr\right)dxd\tau\\
&=\int_0^t\int_{\R^3}\rho_t\left(\int_{\rho_s}^{\rho}r^{-1}P'(r)dr\right)dxd\tau=\int_0^t\int_{\R^3}G(\rho)_s dxd\tau=\int_{\R^3}G(\rho)dx\Big|_0^t.
\end{align*}
Putting the above into \eqref{energy balance law}, the estimate \eqref{L2 energy bound} follows.
\end{proof}

Next, we recall the following estimates on the effective viscous flux $F$ and vorticity $\omega$ for all $t\in[0,T]$. A proof can be found in \cite{LM11} using the Poisson equations \eqref{elliptic eqn for F} and \eqref{elliptic eqn for omega} and the Marcinkiewicz multiplier theorem.

\begin{lem}
Assume that $\rho$ satisfies the bounds in \eqref{pointwise bound on rho}. Then for $t\in[0,T]$ and $r_1,r_2\in(1,\infty)$,
\begin{align}\label{bound on F and omega in terms of u}
\|\nabla F\|_{L^{r_1}}+\|\nabla\omega\|_{L^{r_1}}&\le C(\|\dot{u}\|_{L^{r_1}}+||\nabla u||_{L^{r_1}}+\|(\rho-\trho)^2\|_{L^{r_1}}),\\
\label{bound on u in terms of F and omega}
||\nabla u||_{L^{r_2}}&\le C(|| F||_{L^{r_2}}+||\omega||_{L^{r_2}}+||(\rho-\rho_s)||_{L^{r_2}}).
\end{align} 
\end{lem}

To proceed further, we have to obtain higher order estimates on $u$. In order to obtain better estimates on $\rho$, we subdivide the estimates into two cases namely $t\in[0,1]$ and $t\in(1,T]$. These will be illustrated in Subsection~\ref{small time estimates} and Subsection~\ref{large time estimates} as follows:

\subsection{Estimates on $u$ and $\rho$ for $t\in[0,1]$}\label{small time estimates} 

In this subsection, we aim at estimating $\Phi_1(t)$, $\Phi_2(t)$ and $\Phi_3(t)$ when $t\in[0,1]$. We first derive bounds on $u$ in $L^\infty([0,t];H^1(\R^3))$ under some smallness conditions on $C_0$, which gives a bound on $\Phi_2(t)$.

\begin{lem}\label{H1 estimates}
Assume that $\rho$ satisfies \eqref{pointwise bound on rho} and $C_0\ll1$. For $t\in[0,1]$ and $\alpha\in[0,1]$, we have
\begin{align}\label{H1 bound}
\sup_{0\le \tau\le t}\tau^{1-\tau}\intox|\nabla u|^2dx+\int_0^t\intox \tau^{1-\alpha}|\dot{u}|^2dxd\tau\le CC_0.
\end{align}
\end{lem}

\begin{proof}
We apply the interpolation techniques given by Hoff \cite{hoff02} or Suen \cite{suen20b}. We define differential operators $\mathcal{L}$ acting on functions $w:\R^3\times[0,\infty)\rightarrow\R^3$ by
\begin{align*}
(\mathcal{L} w)^j&=(\rho w^j)_t+\divv(\rho w^j u)-(\mu \Delta u^j +
\lambda \, \divv(u_{x_j})).
\end{align*}
Then we define $w_1$ and $w_2$ by
\begin{align}\label{eqn for w and L}
\left\{ \begin{array}{l}
\mathcal{L} w_1=0,\,\,\,w_1(x,0)=w_{10}(x),\\
\mathcal{L} w_2=-\rho \Big(\rho^{-1}\nabla P-\rho_s^{-1}\nabla P_s\Big),\,\,\,w_2(x,0)=0
\end{array}\right.
\end{align}
for a given $w_{10}$, and if $w_{10}=u_0$, then $w_1+w_2=u$. Following the proof of the energy balance law for the bound \eqref{L2 energy bound}, we readily have
\begin{align}\label{L2 estimate on w1}
&\sup_{0\le \tau\le t}\intox|w_1(x,\tau)|^2dx+\intoxst|\nabla w_1|^2dxd\tau\le C\intox|w_{10}|^2dx,
\end{align}
as well as
\begin{align}\label{L2 estimate on w2}
\sup_{0\le \tau\le t}\intox|w_2(x,\tau)|^2dx&+\intoxst|\nabla w_2|^2dxd\tau\\
&\le Ct\sup_{0\le \tau\le t}\|(\rho-\rho_s)(\cdot,\tau)\|_{L^2}^2.\notag
\end{align}
On the other hand, for $k=0,1$, we multiply equations \eqref{eqn for w and L} for $w_1$ and $w_2$ by $\tau^k\dot{w}_1$ and $\tau^k\dot{w}_2$ respectively and integrate to obtain
\begin{align}\label{H1 estimate on w1}
&\tau^k\intox|\nabla w_1(x,\tau)|^2dx\Big|_{\tau=0}^{\tau=t}+\intoxst\tau^k|\dot{w_1}|^2dxd\tau\\
&\le C\intoxst k\tau^{k-1}|\nabla w_1|^2dxd\tau+C\intoxst\tau^{\frac{3k}{2}}(|\nabla w_1|^3+|\nabla u|^3)dxd\tau,\notag
\end{align}
and
\begin{align}\label{H1 estimate on w2}
&\tau^k\intox|\nabla w_2(x,\tau)|^2dx\Big|_{\tau=0}^{\tau=t}+\intoxst\tau^k|\dot{w_2}|^2dxd\tau\notag\\
&\le C\Big|\intox(P-P_s)\divv (w_2)(x,\tau)dx\Big|_{\tau=0}^{\tau=t}\Big|+C\intoxst\tau^k|\rho-\rho_s||\dot{w_2}|dxd\tau\\
&\qquad+C\intoxst\tau^{\frac{3k}{2}}(|\nabla w_2|^3+|\nabla u|^3)dxd\tau.\notag
\end{align}
The term $\dis C\intoxst k\tau^{k-1}|\nabla w_1|^2dxd\tau$ can be bounded by $C\|w_{10}\|_{L^2}^2$ with the help of \eqref{L2 estimate on w1}, and using \eqref{L2 estimate on w1}-\eqref{L2 estimate on w2}, the term involving $P$ can be bounded by
\begin{align*}
&C\Big|\intox(P-P_s)\divv (w_2)(x,\tau)dx\Big|_{\tau=0}^{\tau=t}\Big|\\
&\le C\Big(\intox|\rho-\rho_s|^2(x,t)dx\Big)^\frac{1}{2}\Big(\intox|\nabla w_2|^2(x,t)dx\Big)^\frac{1}{2}\\
&\le CC_0^\frac{1}{2}\Big(\intox|\nabla w_2|^2(x,t)dx\Big)^\frac{1}{2} .
\end{align*}
Moreover, using the bound \eqref{L2 energy bound} on $\rho-\rho_s$ and the assumption that $t\in[0,1]$, we have
\begin{align*}
&C\intoxst\tau^k|\rho-\rho_s||\dot{w_2}|dxd\tau\\
&\le \Big(\sup_{0\le\tau t}\intox|\rho-\rho_s|^2dx\Big)^\frac{1}{2}\Big(\intoxst\tau^k|\dot{w_2}|^2dxd\tau\Big)^\frac{1}{2}\\
&\le CC_0^\frac{1}{2}\Big(\intoxst\tau^k|\dot{w_2}|^2dxd\tau\Big)^\frac{1}{2}.
\end{align*}
Next we consider the term $\dis \intoxst\tau^{\frac{3k}{2}}|\nabla u|^3dxd\tau$. By applying \eqref{bound on F and omega in terms of u}-\eqref{bound on u in terms of F and omega} and \eqref{L2 energy bound}, for $t\in[0,1]$, we have
\begin{align*}
\intoxst \tau^{\frac{3k}{2}}|\nabla u|^3dxd\tau&\le C\intoxst \tau^{\frac{3k}{2}}(|F|^3+|\omega|^3+|\rho-\rho_s|^3)dxd\tau\\
&\le C\int_0^t \tau^{\frac{3k}{2}}\Big(\intox|F|^2dx\Big)^\frac{3}{4}\Big(\intox|\nabla F|^2dx\Big)^\frac{3}{4}d\tau\\
&\qquad+C\int_0^t \tau^{\frac{3k}{2}}\Big(\intox|\omega|^2dx\Big)^\frac{3}{4}\Big(\intox|\nabla \omega|^2dx\Big)^\frac{3}{4}d\tau+CC_0\\
&\le CC_0^\frac{1}{4}\Big(\sup_{0\le \tau\le t}\tau^{k}\intox|\nabla u|^2dx\Big)^\frac{1}{2}\Big(\intoxst \tau^{k}|\dot{u}|^2dxd\tau\Big)^\frac{3}{4}\\
&\qquad+C_0^\frac{3}{4}\Big(\intoxst \tau^{k}|\dot{u}|^2dxd\tau\Big)^\frac{3}{4}+CC_0.
\end{align*}
Hence under suitable smallness conditions on $C_0$, the term $\dis \intoxst\tau^{\frac{3k}{2}}|\nabla u|^3dxd\tau$ can be absorbed into the left sides of \eqref{H1 estimate on w1} and \eqref{H1 estimate on w2}. Treating the terms $\dis \intoxst\tau^{\frac{3k}{2}}|\nabla w_1|^3dxd\tau$ and $\dis \intoxst\tau^{\frac{3k}{2}}|\nabla w_2|^3dxd\tau$ in a similar way, we can conclude that
\begin{align}
\sup_{0\le \tau\le t}\intox|\nabla w_1|^2dx+\intoxt |\dot{w_1}|^2dxd\tau\le C\|w_{10}\|^2_{H^1},\label{H1 bound on w1 smooth}\\
\sup_{0\le \tau\le t}\tau\intox|\nabla w_1|^2dx+\intoxt \tau|\dot{w_1}|^2dxd\tau\le C\|w_{10}\|^2_{L^2},\label{H1 bound on w1 not smooth}\\
\sup_{0\le \tau\le t}\intox|\nabla w_2|^2dx+\intoxt|\dot{w_2}|^2dxd\tau\le CC_0.\label{H1 bound on w2}
\end{align}
Since the operator $\mathcal{L}$ is linear, we can apply Riesz-Thorin interpolation to deduce from \eqref{H1 bound on w1 smooth}-\eqref{H1 bound on w1 not smooth} that for $t\in[0,1]$ and $\alpha\in[0,1]$
\begin{align}\label{Hs bound on w1}
\sup_{0\le \tau\le t}\tau^{1-\tau}\intox|\nabla u|^2dx+\int_0^t\intox \tau^{1-\alpha}|\dot{u}|^2dxd\tau\le C\|w_{10}\|^2_{H^s}.
\end{align}
By taking $w_{10}=u_0$ in \eqref{Hs bound on w1}, we conclude from \eqref{H1 bound on w2} and \eqref{Hs bound on w1} that the bound \eqref{H1 bound} holds for $t\in[0,1]$.
\end{proof}

\begin{rem}
The result \eqref{H1 bound} from Lemma~\ref{H1 estimates} holds for all $\alpha\in[0,1]$, in particular it holds for $\alpha\in(1,\frac{1}{2}]$.
\end{rem}

We further derive preliminary bounds for $\dot u$ in $L^\infty([0,t];L^2(\R^3))$ when $t\in[0,1]$. Notice that we require $\alpha\in(\frac{1}{2},1]$ on the time layer factor $\tau^{2-\alpha}$ due to the lack of integrability in time near $t=0$ for $\dot u$.

\begin{lem}\label{H2 estimates}
Assume that $\rho$ satisfies \eqref{pointwise bound on rho} and $C_0\ll1$. For $t\in[0,1]$ and $\alpha\in(\frac{1}{2},1]$, we have
\begin{align}\label{H2 bound}
\sup_{0\le \tau\le t}\tau^{2-s}\intox|\dot{u}|^2dx+\int_0^t\intox \tau^{2-\alpha}|\nabla\dot{u}|^2dxd\tau\le CC_0^{\theta},
\end{align}
for some $\theta>0$.
\end{lem}

\begin{proof}
We use the method given in \cite{LM11} with some modifications. We rewrite the momentum equation \eqref{NS}$_2$ as follows.
\begin{align}\label{another form of momentum eqn}
\rho\dot u-\mu\Delta u-\lambda\nabla(\divv(u))=-\nabla(P-P_s)+(\rho-\rho_s)\nabla\Psi.
\end{align}
We apply the material derivative $\dis\frac{D}{Dt}(\cdot)$ on \eqref{another form of momentum eqn}, multiply by $\tau^{2-\alpha}$ and integrate to give
\begin{align}\label{eqn for dot u}
&\int_0^t\intox\tau^{2-\alpha}\rho\frac{D}{Dt}\frac{|\dot{u}|^2}{2}dxd\tau-\mu\intoxst\tau^{2-\alpha}\dot{u}^j(\Delta u^j_t+\divv(\Delta u^j u))dxd\tau\notag\\
&\qquad-\lambda\intoxst\tau^{2-\alpha}\dot{u}^j(\divv(u_{x_j})_t+\divv(u\divv(u_{x_j})))dxd\tau\notag\\
&=-\intoxst\tau^{2-\alpha}\dot{u}^j(P_{x_j t}+\divv(u(P-P_s)_{x_j}))dxd\tau\\
&\qquad+\intoxst\tau^{2-\alpha}\cdot{u}^j(\Psi_{x_j}\rho_t+\divv(u(\rho-\rho_s)\Psi_{x_j}))dxd\tau.\notag
\end{align}
The left side of \eqref{eqn for dot u} can be bounded below by the terms
\begin{align*}
&\frac{\tau^{2-\alpha}}{2}\intox\rho|\dot{u}|^2(x,\tau)dx+\mu\intoxst\tau^{2-\alpha}|\nabla\dot{u}|^2dxd\tau\\
&+\lambda\intoxst\tau^{2-\alpha}(\divv(\dot{u}))^2dxd\tau-C\intoxst\tau^{2-\alpha}|\nabla\dot{u}||\nabla u|^2dxd\tau\\
&-C\intoxst\tau^{1-\alpha}|\dot{u}|^2dxd\tau.
\end{align*}
For the first term on the right side of \eqref{eqn for dot u}, it can be bounded by
\begin{align*}
&\Big|\intoxst\tau^{2-\alpha}\dot{u}^j(P_{x_j t}+\divv(u(P-P_s)_{x_j}))dxd\tau\Big|\\
&\le C\Big(\intoxst\tau^{2-\alpha}|\nabla\dot{u}|^2dxd\tau\Big)^\frac{1}{2}\\
&\qquad\times\Big(\intoxst|\nabla u|^2dxd\tau+\int_0^t\Big(\intox|u|^6dx\Big)^\frac{1}{3}\Big(\intox|\nabla\rho_s|^3dx\Big)^\frac{2}{3}d\tau\Big)^\frac{1}{2}\\
&\le C\Big(\intoxst\tau^{2-\alpha}|\nabla\dot{u}|^2dxd\tau\Big)^\frac{1}{2}\Big(\intoxst|\nabla u|^2dxd\tau\Big)^\frac{1}{2},
\end{align*}
and the second term on the right side of \eqref{eqn for dot u} is bounded by
\begin{align*}
&\Big|\intoxst\tau^{2-\alpha}\cdot{u}^j(\Psi_{x_j}\rho_t+\divv(u(\rho-\rho_s)\Psi_{x_j}))dxd\tau\Big|\\
&\le C\Big(\intoxst\tau^{2-\alpha}(|\nabla\dot{u}|^2+|\dot{u}|^2)dxd\tau\Big)^\frac{1}{2}\\
&\qquad\times\Big(\int_0^t\Big(\intox|u|^6dx\Big)^\frac{1}{3}\Big(\intox(|\Psi|^3+|\nabla\Psi|^3)dx\Big)^\frac{2}{3}d\tau\Big)^\frac{1}{2}\\
&\le C\Big(\intoxst\tau^{2-\alpha}(|\nabla\dot{u}|^2+|\dot{u}|^2)dxd\tau\Big)^\frac{1}{2}\Big(\intoxst|\nabla u|^2dxd\tau\Big)^\frac{1}{2}.
\end{align*}
Hence we obtain from \eqref{eqn for dot u} that
\begin{align}\label{H2 bound step 1}
&\tau^{2-\alpha}\intox\rho|\dot{u}|^2(x,\tau)dx+\intoxst\tau^{2-\alpha}|\nabla\dot{u}|^2dxd\tau\notag\\
&\le C\intoxst\tau^{2-\alpha}|\nabla\dot{u}||\nabla u|^2dxd\tau+C\intoxst\tau^{1-\alpha}|\dot{u}|^2dxd\tau\\
&\qquad+C\Big(\intoxst\tau^{2-\alpha}(|\nabla\dot{u}|^2+|\dot{u}|^2)dxd\tau\Big)^\frac{1}{2}\Big(\intoxst|\nabla u|^2dxd\tau\Big)^\frac{1}{2}.\notag
\end{align}
In view of the bounds given by \eqref{L2 energy bound} and \eqref{H1 bound}, the terms $\dis\intoxst\tau^{1-\alpha}|\dot{u}|^2dxd\tau$ and $\dis\intoxst|\nabla u|^2dxd\tau$ can be bounded by $CC_0$, and by the Cauchy-Schwarz inequality, the term $\dis\intoxst\tau^{2-\alpha}|\nabla\dot{u}||\nabla u|^2dxd\tau$ is bounded by
\begin{align*}
\Big(\intoxst\tau^{2-\alpha}|\nabla\dot{u}|^2dxd\tau\Big)^\frac{1}{2}\Big(\intoxst\tau^{2-\alpha}|\nabla u|^4dxd\tau\Big)^\frac{1}{2}.
\end{align*}
It remains to consider the space-time integral of $\tau^{2-\alpha}|\nabla u|^4$. Using \eqref{bound on u in terms of F and omega},
\begin{align}\label{estimate on L4 of nabla u} 
\intoxst\tau^{2-\alpha}|\nabla u|^4dxd\tau\le C\intoxst\tau^{2-\alpha}(|F|^4+|\omega|^4+|\rho-\rho_s|^4)dxd\tau.
\end{align}
Since $t\in[0,1]$, with the help of \eqref{L2 energy bound}, we have
\begin{align*}
\intoxst\tau^{2-\alpha}|\rho-\rho_s|^4dxd\tau\le CC_0.
\end{align*}
For the integral of $F$ in \eqref{estimate on L4 of nabla u}, using the Sobolev inequality \eqref{GN1} and the estimate \eqref{bound on F and omega in terms of u} on $F$, for $\alpha\in(\frac{1}{2},1]$,
\begin{align*}
&\intoxst\tau^{2-\alpha}|F|^4dxd\tau\le C\int_0^t\tau^{2-\alpha}\Big(\intox|F|^2dx\Big)^\frac{1}{2}\Big(\intox|\nabla F|^2dx\Big)^\frac{1}{2}d\tau\\
&\le C\Big(\sup_{0\le\tau\le t}\tau^{1-\alpha}\intox|F|^2dx\Big)^\frac{1}{2}\Big(\sup_{0\le\tau\le t}\tau^{2-\alpha}\intox|\nabla F|^2dx\Big)^\frac{1}{2}\\
&\qquad\times\Big(\intoxst\tau^{1-\alpha}|\nabla F|^2dxd\tau\Big)\\
&\le CC_0^\frac{1}{2}\Big(\sup_{0\le\tau\le t}\tau^{2-\alpha}\intox|\dot{u}|^2dx+\intox|\nabla u|^2+\intox|\rho-\rho_s|^4dx\Big)^\frac{1}{2}\\
&\qquad\times(\Phi_2(t)+C_0)\\
&\le CC_0^\frac{3}{2}\Big(\sup_{0\le\tau\le t}\tau^{2-\alpha}\intox|\dot{u}|^2dx+C_0\Big)^\frac{1}{2},
\end{align*}
where we have used the bounds \eqref{pointwise bound on rho}, \eqref{L2 energy bound} and \eqref{H1 bound} and the fact that $\frac{1-\alpha}{2}+\frac{2-\alpha}{2}+1-\alpha\le 2-\alpha$ for $\alpha\in(\frac{1}{2},1]$. The integral of $\omega$ in \eqref{estimate on L4 of nabla u} can be treated in a similar way, and we obtain that
\begin{align*}
\intoxst\tau^{2-\alpha}|\nabla u|^4dxd\tau\le CC_0^\frac{3}{2}\Big(\sup_{0\le\tau\le t}\tau^{2-\alpha}\intox|\dot{u}|^2dx+C_0\Big)^\frac{1}{2}+CC_0.
\end{align*}
Therefore, we conclude from \eqref{H2 bound step 1} that
\begin{align*}
\sup_{0\le \tau\le t}\tau^{2-s}\intox|\dot{u}|^2dx+\int_0^t\intox \tau^{2-\alpha}|\nabla\dot{u}|^2dxd\tau\le CC_0^3+CC_0^2+CC_0,
\end{align*}
and the bound \eqref{H2 bound} follows.
\end{proof}

Combining the above results, we now have the following lemma which gives the bound for $\Phi(t)$ when $t\in[0,1]$.

\begin{lem}\label{a priori bounds with rho bounded}
Assume that the hypotheses and notations of Theorem~\ref{a priori bounds} are in force. Assume that $\rho$ satisfies \eqref{pointwise bound on rho} and $C_0\ll1$. Then for any $t\in[0,1]$,
\begin{align}\label{bound on Phi small time}
\Phi(t)&\le CC_0^\theta.
\end{align}
\end{lem}
\begin{proof}
The bound \eqref{bound on Phi small time} follows immediately from \eqref{L2 energy bound}, \eqref{H1 bound} and \eqref{H2 bound} with the smallness assumption $C_0\ll1$.  
\end{proof}

We end this subsection by giving the following auxiliary $L^q$ estimates on the velocity $u$ when $q>6$. It will be used for obtaining pointwise bound on $\rho$ later.

\begin{lem}\label{Lq bound on u lem}
 Assume that the hypotheses and notations of Theorem~\ref{a priori bounds} are in force. Then for $q>6$ satisfying \eqref{condition on q and viscosity} and $t\in[0,1]$,
\begin{align}\label{Lq bound on u}
&\sup_{0\le \tau\le t}\int_{\R^3}|u(x,\tau)|^qdx+\int_0^t\int_{\R^3}|u|^{q-2}|\nabla u|^2dxd\tau+\int_0^t\int_{\R^3}|u|^qdxd\tau\\
&\qquad\le C\left[C_0+\int_{\R^3}|u_0|^qdx\right].\notag
\end{align}
\end{lem}
\begin{proof}
The proof can be found in \cite{suen16}. We point out that the condition \eqref{condition on q and viscosity} on $q$, $\mu$ and $\lambda$ is used for controlling the integral $\dis\int_0^t\int_{\R^3}|u|^{q-2}|\nabla u|^2dxd\tau$ in \eqref{Lq bound on u}.
\end{proof}

\begin{rem}
For the more general case when $t\in[0,T]$, one can obtain
\begin{align*}
&\sup_{0\le \tau\le T}\int_{\R^3}|u(x,\tau)|^qdx+\int_0^T\int_{\R^3}|u|^{q-2}|\nabla u|^2(x,\tau)dxd\tau+\int_0^T\int_{\R^3}|u|^qdxd\tau\\
&\qquad\le C(T)\left[C_0+\int_{\R^3}|u_0|^qdx\right],\notag
\end{align*}
where $C(T)>0$ further depends on $T$. Hence it gives $\dis\int_0^T\int_{\R^3}|u|^qdxd\tau<\infty$, which shows the bound \eqref{condition on weak sol5} for $u$ in $L^q$ when $q>6$.
\end{rem}

\subsection{Estimates on $u$ and $\rho$ for $t>1$}\label{large time estimates}

In this subsection, we obtain estimates on $u$ and $\rho$ when $t>1$. Specifically, we define\begin{align*}
\tilde\Phi_1(t)&=\sup_{1\le \tau\le t}(\|u(\cdot,\tau)\|_{L^2}^2+||(\rho-\rho_s)(\cdot,\tau)||_{L^2}^2) + \int_{0}^{t}\|\nabla u(\cdot,\tau)\|_{L^2}^2d\tau,\\
\tilde\Phi_2(t)&=\sup_{1\le \tau\le t}\|\nabla u(\cdot,\tau)\|^2_{L^2}+\int_1^t\|\dot{u}(\cdot,\tau)\|_{L^2}^2d\tau,\\
\tilde\Phi_3(t)&=\sup_{1\le \tau\le t}\|\dot{u}(\cdot,\tau)\|^2_{L^2}+\int_1^t\|\nabla\dot{u}(\cdot,\tau)\|^2_{L^2}d\tau,
\end{align*}
\begin{align}\label{def of Phi}
\tilde\Phi(t)=\tilde\Phi_1(t)+\tilde\Phi_2(t)+\tilde\Phi_3(t),
\end{align}
and we aim at bounding $\tilde\Phi(t)$ under the assumption that the pointwise bounds in \eqref{pointwise bound on rho} hold for the density $\rho$. As inspired by the work \cite{LM11}, we further introduce the following auxiliary functionals $\mathcal{H}(t)$ and $\mathcal{A}(t)$ given by
\begin{align*}
\mathcal{H}(t)&=\int_1^t\intox|\nabla u|^4dxd\tau,\\
\mathcal{A}(t)&=\int_1^t\intox|\rho-\rho_s|^\frac{10}{3}dxd\tau.
\end{align*}

To begin with, by the energy balance law given by Lemma~\ref{L2 estimate lem}, for all $t>1$, we readily have
\begin{align}\label{L2 estimate large time}
\tilde\Phi_1(t)\le CC_0.
\end{align}
In the next lemma, we will give the bounds on $\tilde\Phi_2$ and $\tilde\Phi_3$ in terms of $\mathcal{H}$ and $\mathcal{A}$.

\begin{lem}
Assume that the hypotheses and notations of Theorem~\ref{a priori bounds} are in force. Suppose that $\rho$ satisfies the pointwise bounds \eqref{pointwise bound on rho}. Then for any $t>1$,
\begin{align}
\tilde\Phi_2(t)&\le C\Big(C_0+\mathcal{A}(t)+C_0^\frac{1}{2}\mathcal{H}(t)^\frac{1}{2}\Big),\label{estimate on Phi2 large time}\\
\tilde\Phi_3(t)&\le C(C_0^\theta+\tilde\Phi_2(t)+\mathcal{H}(t)).\label{estimate on Phi3 large time}
\end{align}
\end{lem}

\begin{proof}
First of all, following steps given in \cite[pp. 501]{LM11}, we can rewrite the momentum equation \eqref{NS}$_2$ as follows.
\begin{align}\label{another form of momentum eqn 2}
&\rho\dot{u}-\mu\Delta u-\lambda\nabla(\divv(u))+\rho_s\nabla(\rho_s^{-1}(P-P_s))\notag\\
&+(\rho-\rho_s)^2\rho_s^{-1}\nabla\rho_s\int_0^1\int_0^1\xi P''(\rho_s+\xi\eta(\rho-\rho_s)d\xi d\eta=0.
\end{align}
Multiply \eqref{another form of momentum eqn 2} by $\dot{u}$, integrate the resulting equation and follow the analysis given in \cite{hoff95} and \cite{LM11}, for $t>1$, we have
\begin{align}\label{H1 estimate large time step 1}
&\intox|\nabla u|^2(x,t)dx+\int_1^t\intox|\dot{u}|^2dxd\tau\notag\\
&\le C\int_1^t\intox|\nabla u|^3dxd\tau+\tilde\Phi_1(t)+\Phi_1(1)\\
&\qquad+\Big|\int_1^t\intox\dot{u}(\rho-\rho_s)^2\rho_s^{-1}\nabla\rho_s\int_0^1\int_0^1\xi P''(\rho_s+\xi\eta(\rho-\rho_s)d\xi d\eta dxd\tau\Big|\notag\\
&\qquad+\Big|\intox\divv(u)(P-P_s)dx\Big|^t_{1}\Big|+\Big|\intox u\cdot\nabla\rho_s\rho_s^{-1}(P-P_s)dx\Big|^t_{1}\Big|.\notag
\end{align}
The first term on the right side of \eqref{H1 estimate large time step 1} can be readily bounded by
\begin{align*}
\Big(\int_1^t\intox|\nabla u|^2dxd\tau\Big)^\frac{1}{2}\Big(\int_1^t\intox|\nabla u|^4dxd\tau\Big)\le \tilde\Phi_1(t)^\frac{1}{2}\mathcal{H}(t)^\frac{1}{2},
\end{align*}
and using the bounds \eqref{bound on Phi small time} and \eqref{L2 estimate large time}, the last two terms on the far right side of \eqref{H1 estimate large time step 1} are bounded by
\begin{align*}
\Big(\sup_{1\le t\le 1}(\|u(\cdot,t)\|_{L^2}+\|\nabla u(\cdot,t)\|_{L^2})\Big)\tilde\Phi_1(t)^\frac{1}{2}\le (\tilde\Phi_1(t)^\frac{1}{2}+\tilde\Phi_2(t)^\frac{1}{2})\tilde\Phi_1(t)^\frac{1}{2}.
\end{align*}
To estimate the fourth term on the right side of \eqref{H1 estimate large time step 1}, using the pointwise bound \eqref{pointwise bound on rho} on $\rho$, 
\begin{align*}
&\Big|\int_1^t\intox\dot{u}(\rho-\rho_s)^2\rho_s^{-1}\nabla\rho_s\int_0^1\int_0^1\xi P''(\rho_s+\xi\eta(\rho-\rho_s)d\xi d\eta dxd\tau\Big|\\
&\le C\Big(\int_1^t\intox|\dot{u}|^2dxd\tau\Big)^\frac{1}{2}\Big(\int_1^t\intox(\rho-\rho_s)^4dxd\tau\Big)^\frac{1}{2}\\
&\le C\tilde\Phi_2(t)^\frac{1}{2}\mathcal{A}(t)^\frac{1}{2}.
\end{align*}
Therefore, together with the bound \eqref{L2 estimate large time} on $\tilde\Phi_1(t)$ and the bound \eqref{L2 energy bound} on $\Phi_1(1)$, we conclude that \eqref{estimate on Phi2 large time} holds for $t>1$.

Next, to prove the bound \eqref{estimate on Phi3 large time}, we perform the similar analysis given in the proof of Lemma~\ref{H2 estimates} and obtain, for $t>1$, 
\begin{align*}
&\intox|\dot{u}|^2(\cdot,t)dx+\int_1^t|\nabla\dot{u}|^2(\cdot,\tau)dxd\tau\\
&\le C\int_1^t\intox|\dot{u}|^2dxd\tau+C\int_1^t\intox|\nabla\dot{u}||\nabla u|^2dxd\tau\\
&\qquad+C\Big(\int_1^t\intox(|\nabla\dot{u}|^2+|\dot{u}|^2)dxd\tau\Big)^\frac{1}{2}\Big(\int_1^t\intox|\nabla u|^2dxd\tau\Big)^\frac{1}{2}+\Phi(1).
\end{align*}
By the Cauchy-Schwarz inequality,
\begin{align*}
\int_1^t\intox|\nabla\dot{u}||\nabla u|^2dxd\tau&\le\Big(\int_1^t\intox|\nabla\dot{u}|^2dxd\tau\Big)^\frac{1}{2}\Big(\int_1^t\intox|\nabla u|^4dxd\tau\Big)^\frac{1}{2}\\
&\le \tilde\Phi_3(t)^\frac{1}{2}\mathcal{H}(t)^\frac{1}{2},
\end{align*}
hence the bound \eqref{estimate on Phi3 large time} follows with the help of the bound \eqref{L2 estimate large time} on $\tilde\Phi_1(t)$ and the bound \eqref{bound on Phi small time} on $\Phi(1)$.
\end{proof}

We close the estimates on $\tilde\Phi_2$ and $\tilde\Phi_3$ by bounding $\mathcal{A}(t)$ and $\mathcal{H}(t)$ in terms of $\tilde\Phi$, which will be illustrated in the following lemma:

\begin{lem}
Assume that the hypotheses and notations of Theorem~\ref{a priori bounds} are in force. Suppose that $\rho$ satisfies the pointwise bounds \eqref{pointwise bound on rho}. Then for any $t>1$,
\begin{align}
\mathcal{H}(t)&\le C(\tilde\Phi(t)^2+\mathcal{A}(t)^2+\mathcal{A}(t)+C_0^2),\label{estimate on H large time}\\
\mathcal{A}(t)&\le C\tilde\Phi(t)^\frac{2}{3}(\tilde\Phi(t)+\mathcal{A}(t)+C_0)+CC_0^\frac{5}{3}+CC_0.\label{estimate on A large time}
\end{align}
\end{lem}

\begin{proof}
To prove the bound \eqref{estimate on H large time} on $\mathcal{H}$, using \eqref{bound on u in terms of F and omega},
\begin{align}\label{estimate on L4 of nabla u large time} 
\int_1^t\intox|\nabla u|^4dxd\tau\le C\int_1^t\intox(|F|^4+|\omega|^4+|\rho-\rho_s|^4)dxd\tau.
\end{align}
The integral of $|\rho-\rho_s|^4$ can be bounded by $C\mathcal{A}(t)$, and by the Sobolev inequality \eqref{GN1}, the estimate \eqref{bound on F and omega in terms of u} on $F$ and the bound \eqref{L2 estimate large time} on $\tilde\Phi_1(t)$, we also have
\begin{align*}
&\int_1^t\intox|F|^4dxd\tau\\
&\le\Big(\sup_{1\le\tau\le t}\intox|F|^2dx\Big)^\frac{1}{2}\Big(\sup_{1\le\tau\le t}\intox|\nabla F|^2dx\Big)^\frac{1}{2}\Big(\int_1^t\intox|\nabla F|^2dxd\tau\Big)\\
&\le C\Big(\tilde\Phi_2(t)+C_0\Big)^\frac{1}{2}\Big(\tilde\Phi_3(t)+C_0\Big)^\frac{1}{2}(\tilde\Phi_2(t)+\mathcal{A}(t)+C_0).
\end{align*}
We estimate the integral of $|\omega|^4$ in a similar way and we conclude from \eqref{estimate on L4 of nabla u large time} that
\begin{align*}
\int_1^t\intox|\nabla u|^4dxd\tau\le C(\tilde\Phi(t)+C_0)(\tilde\Phi(t)+\mathcal{A}(t)+C_0)+C\mathcal{A}(t),
\end{align*}
which gives \eqref{estimate on H large time}. To prove \eqref{estimate on A large time}, we make use of the definition of $F$ and rewrite the mass equation \eqref{NS}$_1$ as follows:
\begin{align}\label{mass eqn rewritten}
(\mu+\lambda)\frac{D}{Dt}(\rho-\rho_s)+\rho(P-P_s)=-\rho\rho_s F-(\mu+\lambda)u\cdot\nabla\rho_s.
\end{align}
We multiply \eqref{mass eqn rewritten} by $\sgn(\rho-\rho_s)|\rho-\rho_s|^\frac{7}{3}$, integrate and use the pointwise bound \eqref{pointwise bound on rho} on $\rho$ to obtain
\begin{align}\label{mass eqn rewritten integral ineq}
&\int_1^t\intox|\rho-\rho_s|^\frac{10}{3}dxd\tau\notag\\
&\le C\Big|\intox|\rho-\rho_s|^\frac{10}{3}dx\Big|_{1}^t\Big|+C\int_1^t\intox(|F|^\frac{10}{3}+|u|^\frac{10}{3})dxd\tau.
\end{align}
Using the bounds \eqref{bound on Phi small time} and \eqref{L2 estimate large time}, the terms $\dis \Big|\intox|\rho-\rho_s|^\frac{10}{3}dx\Big|_{1}^t\Big|$ are bounded by $CC_0$. To bound the integrals of $|F|^\frac{10}{3}$ and $|u|^\frac{10}{3}$, we use the Sobolev inequality \eqref{GN1}, the estimate \eqref{bound on F and omega in terms of u} and the bound \eqref{L2 estimate large time} to get
\begin{align*}
&\int_1^t\intox(|F|^\frac{10}{3}+|u|^\frac{10}{3})dxd\tau\\
&\le C\int_1^t\Big(\intox|F|^2dx\Big)^\frac{2}{3}\Big(\intox|\nabla F|^2dx\Big)d\tau\\
&\qquad+C\int_1^t\Big(\intox|u|^2dx\Big)^\frac{2}{3}\Big(\intox|\nabla u|^2dx\Big)d\tau\\
&\le C\tilde\Phi(t)^\frac{2}{3}(\tilde\Phi(t)+\mathcal{A}(t)+C_0)+CC_0^\frac{5}{3}.
\end{align*}
Therefore, we conclude that, for $t>1$, 
\begin{align*}
\int_1^t\intox|\rho-\rho_s|^\frac{10}{3}dxd\tau\le C\tilde\Phi(t)^\frac{2}{3}(\tilde\Phi(t)+\mathcal{A}(t)+C_0)+CC_0^\frac{5}{3}+CC_0,
\end{align*}
which implies \eqref{estimate on A large time}.
\end{proof}

We summarise the above results and give the bound for $\tilde\Phi(t)$ when $t>1$.

\begin{lem}\label{a priori bounds with rho bounded large time}
Assume that the hypotheses and notations of Theorem~\ref{a priori bounds} are in force. Assume that $\rho$ satisfies \eqref{pointwise bound on rho} and $C_0\ll1$. Then for any $t>1$,
\begin{align}\label{bound on Phi large time}
\mathcal{H}(t)+\mathcal{A}(t)+\tilde\Phi(t)&\le CC_0^\theta.
\end{align}
\end{lem}
\begin{proof}
The bound \eqref{bound on Phi large time} follows immediately from \eqref{L2 estimate large time}, \eqref{estimate on Phi3 large time}, \eqref{estimate on H large time} and \eqref{estimate on A large time} with the smallness assumption $C_0\ll1$.  
\end{proof}

\subsection{Pointwise bound on $\rho$ and proof of Theorem~\ref{a priori bounds}}\label{Pointwise bound on rho subsection}

We now close the estimates on $(\rho,u)$ by proving the pointwise bounds \eqref{pointwise bound on rho} on $\rho$. First of all, by the bounds \eqref{bound on Phi small time} and \eqref{bound on Phi large time}, we have
\begin{align}\label{bound on Phi all time}
\Phi(t)&\le CC_0^\theta,\qquad t\in[0,T].
\end{align} 

Next, we obtain some higher order estimates on $\|u\|_{L^r}$, $\|u\|_{H^{\alpha}}$, $\langle u\rangle^{\frac{1}{2},\frac{1}{4}}_{\R^3 \times [\tau,\infty)}$ for $\alpha\in[0,1]$, $r\in(3,\frac{3}{1-\alpha})$ and $\tau>0$.

\begin{lem}\label{holder estimates on u lem}
Assume that the hypotheses and notations of Theorem~\ref{a priori bounds} are in force. Suppose that $\rho$ satisfies the pointwise bounds \eqref{pointwise bound on rho}. Then for any $t\in(0,T]$ and $\tau>0$,
\begin{align}\label{bound on u weak sol lem}
\sup_{0\le \tau<t}\|u(\cdot,\tau)\|_{H^{\alpha}}\le CC_0^\theta,\qquad \sup_{0\le \tau<t}\|u(\cdot,\tau)\|_{L^{r}}\le C(r)C_0^\theta,
\end{align}
and
\begin{align}\label{Holder estimates in space time lem}
\langle u\rangle^{\frac{1}{2},\frac{1}{4}}_{\R^3 \times [t,\infty)},\langle F\rangle^{\frac{1}{2},\frac{1}{4}}_{\R^3 \times [t,\infty)},\langle \omega\rangle^{\frac{1}{2},\frac{1}{4}}_{\R^3 \times [t,\infty)}\leq C(\tau)C_{0}^{\theta},\;\;\;t\ge\tau>0,
\end{align}
where $\alpha\in[0,1]$ and $r\in(3,\frac{3}{1-\alpha})$. Here $C(r)$ is a positive constant which depends only on $r$ and $C(\tau)>0$ may depend additionally on a positive lower bound for $\tau$.
\end{lem}
\begin{proof}
To show the estimate for $\|u(\cdot,\tau)\|_{H^{\alpha}}$, we follow the proof of Lemma~\ref{H1 estimates} to obtain, for $k=0,1$ and $t\le1$,
\begin{align*}
\sup_{0\le \tau<t}\|w_1(\cdot,\tau)\|_{H^k}\le C\|w_{10}\|_{H^k}.
\end{align*}
Hence by interpolation, we have
\begin{align*}
\sup_{0\le \tau<t}\|w_1(\cdot,\tau)\|_{H^\alpha}\le C\|w_{10}\|_{H^\alpha}
\end{align*}
for $\alpha\in[0,1]$. Choosing $w_{10}=u_0$, then we have $u=w_1+w_2$ and together with the estimates on $w_2$, we conclude that
\begin{align}\label{holder estimate small time}
\sup_{0\le \tau<t\le1}\|u(\cdot,\tau)\|_{H^\alpha}\le C\|u_{0}\|_{H^\alpha}.
\end{align}
For the case when $t\in(1,T]$, by the bound \eqref{bound on Phi large time}, we readily have
\begin{align*}
\sup_{1<\tau\le t}\intox(|u|^2+|\nabla u|^2)dxd\tau\le CC_0^\theta,
\end{align*}
hence it implies
\begin{align}\label{holder estimate large time}
\sup_{1<\tau\le t}\|u(\cdot,\tau)\|_{H^\alpha}\le CC_0^\theta,
\end{align}
and we have the estimate on $\dis\sup_{0\le \tau<t}\|u(\cdot,\tau)\|_{H^\alpha}$ for all $t\in[0,T]$. The estimate on $\dis\sup_{0\le \tau<t}\|u(\cdot,\tau)\|_{L^r}$ then follows from \eqref{holder estimate small time}-\eqref{holder estimate large time} and the imbedding $H^\alpha\hookrightarrow L^r$ for $r\in(3,\frac{3}{1-\alpha})$. Finally, to show \eqref{Holder estimates in space time lem}, we only consider the case for $F$ since the cases for $u$ and $\omega$ are somewhat simpler. Notice that using \eqref{bound on F and omega in terms of u} for $r_1=6$ and the Sobolev imbedding \eqref{GN1},
\begin{align}\label{L6 bound on nabla F}
\|\nabla F\|_{L^{6}}&\le C(\|\dot{u}\|_{L^{6}}+\|(\rho-\trho)\|_{L^{6}})\\
&\le C(\|\nabla\dot{u}\|_{L^{2}}+\|(\rho-\trho)\|_{L^{2}}).\notag
\end{align}
Fix $\tau>0$, by the bound \eqref{bound on Phi all time}, the right side of \eqref{L6 bound on nabla F} can be bounded in terms of $C_0$ and $\tau$ whenever $t\ge\tau$. Therefore, we have
\begin{align*}
\|F(\cdot,t)\|_{L^\infty}+\langle F(\cdot,t)\rangle^\frac{1}{2}\le CC_0^\theta,\qquad t\ge\tau.
\end{align*}
For the H\"older continuity in time, we fix $x$ and $t_2\ge t_1\ge\tau$ and compute
\begin{align*}
|F (x,t_2)-F (x,t_1)|&\le\frac{1}{|B_{R}(x)|}\int_{B_R(x)}\left |F(z,t_2)-F(z,t_1)\right |dz + C(\tau)C_0^{\theta}R^{\frac{1}{2}}\notag\\
&\le R^{-\frac{3}{2}}|t_2 - t_1|\sup_{t\ge\tau}\left( \int |F_t|^2dx\right)^{\frac{1}{2}} + C(\tau)C_0^{\theta}R^{\frac{1}{2}}\notag\\
&\le C(\tau)C_0^\theta\left[R^{-\frac{3}{2}}|t_2 - t_1|+R^{\frac{1}{2}}\right]
\end{align*}
by the bound \eqref{bound on Phi all time}. Taking $R=|t_2 - t_1|^\frac{1}{2}$ we then obtain the estimate in \eqref{Holder estimates in space time lem} for $F$.
\end{proof}

Now we are ready to establish the pointwise bound \eqref{pointwise bound on rho} on $\rho$ and complete the proof of Theorem~\ref{a priori bounds}. The proof of \eqref{pointwise bound on rho} consists of a maximum-principle argument applied along particle trajectories of $u$, we only sketch it here and details can be found in \cite{suen13b, suen14, suen16}. Fix $y\in\R^3$ and define the corresponding particle trajectory $x(t)$ by
\begin{align*}
\left\{ \begin{array}
{lr} \dot{x}(t)
=u(x(t),t)\\ x(0)=y.
\end{array} \right.
\end{align*}
We have from the definition \eqref{def of F and omega} of $F$ and the mass equation \eqref{NS}$_1$ that
\begin{align*}
(\mu+\lambda)\frac{d}{dt}[\log\rho(x(t),t)]+P(\rho(x(t),t))-P_s=-\rho_sF(x(t),t).
\end{align*}
For $T\le1$, we integrate from $t_0$ to $t_1$ for $t_1,t_2\in[0,T]$, and abbreviate $\rho(x(t),t)$ by $\rho(t)$, etc., we then obtain
\begin{align}\label{integral on F}
(\mu+\lambda)[\log\rho(s)-\log(\rho_s)]\Big |_{ t_0}^{t_1}+\int_{ t_0}^{t_1}[P(\tau)-P_s]d\tau=-\int_{ t_0}^{t_1}\rho_sF(\tau)d\tau.
\end{align}
Since $P$ is increasing, the integral of $P$ on the left side of \eqref{integral on F} is a dissipative term which is harmless. On the other hand, using the Poisson equation \eqref{elliptic eqn for F}, we can rewrite the integral of $F$ as 
\begin{align}\label{integral on rho F}
-\int_{ t_0}^{t_1}\rho_s F(\tau)d\tau&=-\int_{ t_0}^{t_1}\int_{\R^3}\Gamma_{x_j}(x(\tau)-y)\rho\dot{u}^{j}(y,\tau)dyd\tau\\
&\qquad+\int_{ t_0}^{t_1}\int_{\R^3}\Gamma_{x_j}(x(\tau)-y)\left[(P_s)_{x_j}\rho_s ^{-1}(\rho_s -\rho)\right]dyd\tau,\notag
\end{align}
where $\Gamma$ is the fundamental solution of the Laplace operator on $\R^3$. Invoking the method suggested in \cite{hoff95}, the first integral on the right side of \eqref{integral on rho F} can be bounded by
\begin{align}\label{bound on integral on rho F}
\Big|\int_{ t_0}^{t_1}\int_{\R^3}&\Gamma_{x_j}(x(t)-y)\rho\dot{u}^{j}(y,\tau)dyd\tau\Big|\notag\\
&\le||\Gamma_{x_j}*(\rho u^{j})(\cdot,t_1)||_{L^{\infty}}+||\Gamma_{x_j}*(\rho u^{j})(\cdot, t_0)||_{L^{\infty}}\\
&\,\,\,+\left |\int_{ t_0}^{t_1}\int_{\R^3}\Gamma_{x_j x_k}(x(\tau)-y)\left[u^{k}((x(\tau),\tau)-u^{k}(y,\tau)\right](\rho u^{j})(y,\tau)dyd\tau\right |.\notag
\end{align}
To bound the term $||\Gamma_{x_j}*(\rho u^{j})(\cdot,\tau)||_{L^{\infty}}$ for $\tau=t_0,t_1$ appeared in \eqref{bound on integral on rho F}, we make use of the $L^q$ estimate \eqref{Lq bound on u} given by Lemma~\ref{Lq bound on u lem} to obtain that
\begin{align*}
||\Gamma_{x_j}*(\rho u^{j})(\cdot,\tau)||_{L^{\infty}}&\le C(\|u\|_{L^2}+\|u\|_{L^2}^\xi\|u\|_{L^q}^{1-\xi})\\
&\le CC_0^{\theta'},
\end{align*}
for some positive constant $\theta'$ and $\xi$ depends on $q$ with $\xi\in(0,1)$. For the space-time integral on the right side of \eqref{bound on integral on rho F}, we can apply the estimate \eqref{Holder estimates in space time lem} on the semi-H\"{o}lder norm $\langle u(\cdot,\tau)\rangle^\frac{1}{2}$ to bound it in terms of $C_0$. Hence we have
\begin{align*}
\Big|\int_{ t_0}^{t_1}\int_{\R^3}&\Gamma_{x_j}(x(\tau)-y)\rho\dot{u}^{j}(y,\tau)dyd\tau\Big|\le CC_0^{\theta'},
\end{align*}
For the second integral on the right side of \eqref{integral on rho F}, it can be readily estimated as follows: for $\tau\in[t_0,t_1]\subset[0,T]$,
\begin{align*}
&\Big|\int_{\R^3}\Gamma_{x_j}(x(\tau)-y)\left[(P_s)_{x_j}\rho_s ^{-1}(\rho_s -\rho)\right]dy\notag\\
&\le C\int_{\R^3}|y|^{-2}|(\rho-\rho_s)(y,\tau)|^2 dy\Big|\notag\\
&\le C\left[\int_{|y|\le1}|y|^{-\frac{5}{2}}dy\right]^\frac{4}{5}\left[\int_{|y|\le1}|(\rho-\rho_s)(y,\tau)|^5 dy\right]^\frac{1}{5}\notag\\
&\qquad+C\left[\int_{|y|>1}|y|^{-4}dy\right]^\frac{1}{2}\left[\int_{|y|>1}|(\rho-\rho_s)(y,\tau)|^2 dy\right]^\frac{1}{2}\\
&\le C\left[\Phi(\tau)^\frac{1}{5}+\Phi(\tau)^\frac{1}{2}\right]\le CC_0^{\theta'}.
\end{align*}
Hence we conclude that
\begin{align*}
\int_{ t_0}^{t_1}F(\tau)d\tau\le CC_0^{\theta'},
\end{align*}
and by exploiting the smallness condition on $C_0$, we can see that the density $\rho$ should remain inside the interval $[\frac{1}{2}\rho_{2},2\rho_{1}]$ for all $t\in[0,T]$, provided that the initial density satisfies $\rho_0(x)\in[\rho_{2},\rho_{1}]$ for $x\in\R^3$. For case of $T>1$, the proof is just similar and simpler since we can ignore the initial time factor near $t=0$. This completes the proof of Theorem~\ref{a priori bounds}.

\section{Existence and long time behaviour of weak solutions: Proof of Theorem~\ref{existence of weak sol thm}}\label{proof of existence section}

In this section, we complete the proof of Theorem~\ref{existence of weak sol thm} by proving the global-in-time existence and long time behaviour solution to \eqref{NS}. To achieve our goals, we make use of the {\it a priori} estimates derived in Section~\ref{prelim section}. Specifically, we let initial data $(\rho_0,u_0)$ be given satisfying \eqref{smallness assumption} with $u_0\in H^\alpha$ for $\alpha\in(\frac{1}{2},1]$, and we fix those constants $\theta$, $M$ defined in Theorems~\ref{a priori bounds}. Upon choosing $(\rho^\eta_0,u^\eta_0)$ as a smooth approximation of $(\rho_0,u_0)$ which can be obtained by convolving $(\rho_,u_0)$ with the standard mollifying kernel of width $\eta>0$, we can apply the local existence results obtained by Nash \cite{nash} or Tani \cite{tani} to show that there is a smooth local solution $(\rho^\eta,u^\eta)$ of \eqref{NS} with initial data $(\rho^\eta_0,u^\eta_0)$ defined up to a positive time $T$. The {\em a priori} estimates of Theorem~\ref{a priori bounds} and Lemma~\ref{holder estimates on u lem} then apply to show that for $t\in(0,T]$ and $\tau>0$,
\begin{align}\label{bounds on approx sol 1}
\Phi(t)\le MC_{0}^{\theta}
\end{align}
\begin{align}\label{bounds on approx sol 2}
\frac{1}{2}\rho_{2}\le\rho^{\eta}(x,t)\le2\rho_{1}
\end{align}
\begin{align}\label{bounds on approx sol 3}
\sup_{0\le \tau<t}\|u^{\eta}(\cdot,\tau)\|_{H^{\alpha}}\le CC_0^\theta,\qquad \sup_{0\le \tau<t}\|u^{\eta}(\cdot,\tau)\|_{L^{r}}\le C(r)C_0^\theta,
\end{align}
\begin{align}\label{bounds on approx sol 4}
\langle u^{\eta}\rangle^{\frac{1}{2},\frac{1}{4}}_{\R^3 \times [t,\infty)},\langle F^{\eta}\rangle^{\frac{1}{2},\frac{1}{4}}_{\R^3 \times [t,\infty)},\langle \omega^{\eta}\rangle^{\frac{1}{2},\frac{1}{4}}_{\R^3 \times [t,\infty)}\leq C(\tau)C_{0}^{\theta},\;\;\;t\ge\tau,
\end{align}
where $\Phi(t)$ is defined by \eqref{def of Phi} but with $(\rho,u)$ replaced by $(\rho^\eta,u^\eta)$. By the compactness argument given in \cite{suen13b, suen14, suen16}, the bounds \eqref{bounds on approx sol 1}-\eqref{bounds on approx sol 4} will then be sufficient for showing that, there is a sequence $\eta_k\to 0$ and functions $u, \rho$ such that as $k\to\infty$,
\begin{align}\label{strong convergence on u}
\mbox{ $u^{\eta_k}\rightarrow u$ uniformly on compact sets in $\R^3\times (0,\infty)$};\end{align}
\begin{align}\label{weak convergence 1}
\mbox{$\nabla u^{\eta_k}(\cdot,\tau),\nabla\omega^{\eta_k}(\cdot,\tau)\rightharpoonup\nabla u(\cdot,\tau)\nabla\omega(\cdot,\tau)$ weakly in $L^2(\R^3)$
for all $t>0$; }
\end{align}
\begin{align}\label{weak convergence 2}
&\mbox{$\sigma^{\frac{1-\alpha}{2}}\dot{u}^{\eta_k},\sigma^{\frac{2-\alpha}{2}}\nabla\dot{u}^{\eta_k},\rightharpoonup\sigma^{\frac{1-\alpha}{2}}\dot{u},\sigma^{\frac{2-\alpha}{2}}\nabla\dot{u}$ weakly in $L^2(\R^3\times[0,\infty))$;}
\end{align}
and
\begin{align}\label{strong convergence on rho}
\mbox{$\rho^{\eta_k}(\cdot,t)\to \rho(\cdot,t)$ strongly in $L^2_{loc}(\R^3)$ for every $t\ge 0$,} 
\end{align}
where $\sigma(t)=\min\{1,t\}$ and $\alpha\in(\frac{1}{2},1]$. The limiting functions $(\rho,u)$ then inherit the bounds in \eqref{energy bound on weak sol}-\eqref{pointwise bound on rho weak sol}. It is also clear from the modes of convergence described in \eqref{strong convergence on u}-\eqref{strong convergence on rho} that $(\rho,u)$ satisfies the weak forms \eqref{WF1}-\eqref{WF2} of \eqref{NS}.  

To prove that $\rho$ is piecewise $C^{\beta(t)}$ on $[0,T]$ for the case when $\rho_0$ is piecewise $C^{\beta_0}$, it involves an argument which is based on the observation of ``enhanced regularity" gained by the effective viscous flux $F$, where $F$ is given by \eqref{def of F and omega} with $\rho$ and $u$ being the limiting functions in \eqref{strong convergence on u}-\eqref{strong convergence on rho}. Details of the proof can be found in \cite{hoff02, suen20b, suen20mhd} and we only give a sketch here. We recall the decomposition of $u$ which is given by $u=u_{F}+u_{P}$, where $u_{F}$, $u_{P}$ satisfy
\begin{align}\label{decomposition of u}
\left\{
 \begin{array}{lr}
(\mu+\lambda)\Delta (u_{F})^{j}=(\rho_sF)_{x_j} +(\mu+\lambda)(\omega)^{j,k}_{x_k}\\
(\mu+\lambda)\Delta (u_P)^{j}=(P-P_s)_{x_j}.\\
\end{array}
\right.
\end{align}
Using the estimates \eqref{bound on F and omega in terms of u} on $F$ and $\omega$, and together with \eqref{energy bound on weak sol}, we readily have
\begin{align}\label{bound on nabla uF}
\int_{0}^{T}||\nabla u_{F}(\cdot,\tau)||_{\infty}d\tau\le C(T)
\end{align}
for some positive constant $C(T)$. On the other hand, in order to control $u_P$, by applying the results from Bahouri-Chemin \cite{BC94} on Newtonian potential, we can make use of the pointwise bounds \eqref{condition on weak sol4} on $\rho$ to show that $u_P$ is, in fact, log-Lipschitz with bounded log-Lipschitz seminorm. This is sufficient to guarantee that the integral curve $x(y,t)$ as defined by 
\begin{align*}
\left\{ \begin{array}
{lr} \dot{x}(t)
=u(x(t),t)\\ x(0)=y,
\end{array} \right.
\end{align*}
is H\"{o}lder-continuous in $y$. Upon integrating the mass equation along integral curves $x(t,y)$ and $x(t,z)$, subtracting and recalling the definition \eqref{def of F and omega} of $F$, we arrive at 
\begin{align}\label{Holder norm estimate of rho}
\log\rho&(x(T,y),T)-\log\rho(x(T,z),T)\notag\\
&=\log\rho_0(y)-\log\rho_0(z)
+\int_0^T [P(\rho(x(\tau,y),\tau)-P(\rho(x(\tau,z),\tau)]d\tau \\
&\qquad\qquad\qquad\qquad\qquad\qquad\quad+\int_0^t[F(x(\tau,y),\tau)-F(x(\tau,z),\tau)]d\tau.\notag
\end{align}
Since the pressure $P(\cdot)$ is an increasing function, the integral involving $P$ on the right side of the above can be dropped out. Moreover, with the help of the estimate \eqref{bound on F and omega in terms of u} on $F$ and the H\"{o}lder-continuity of $x(y,t)$, the third term can be bounded by $C(T)$. Hence we can conclude from \eqref{Holder norm estimate of rho} that $\rho(\cdot,\tau)$ is $C^{\beta(t)}$ on $[0,T]$ for some $\beta(t)\in(0,\beta_0]$ with bounded modulus.

Finally, it remains to prove the long time behaviour of $\rho$, $u$ and $F$ as described in \eqref{long time 0}-\eqref{long time 3}. To show \eqref{long time 0}, using the bounds \eqref{bounds on approx sol 1}-\eqref{bounds on approx sol 2} and recalling the estimate \eqref{estimate on A large time}, for all $q_1\ge\frac{10}{3}$ and $t>0$, we have
\begin{align*}
\int_0^t\intox|\rho-\rho_s|^{q_1}(x,\tau)dxd\tau\le C\int_0^t\intox|\rho-\rho_s|^\frac{10}{3}(x,\tau)dxd\tau\le CC_0^\theta,
\end{align*}
hence we have
\begin{align*}
\frac{1}{t}\int_0^t\intox|\rho-\rho_s|^{q_1}(x,\tau)dxd\tau\le \frac{CC_0^\theta}{t},
\end{align*}
and \eqref{long time 0} follows by taking $t\to\infty$.

To show \eqref{long time 2}, for $q_2\in[\frac{10}{3},6]$ and $t\ge1$, we define the functional $\mathcal{F}^\eta(t)$ by
\begin{align*}
\mathcal{F}^\eta(t)=\intox|F^\eta(x,t)|^{q_2}dx,
\end{align*}
where $ F^\eta=\rho_s^{-1}(\mu+\lambda)\divv(u)^\eta-\rho_s^{-1}(P(\rho^\eta)-P(\rho_s ))$. From the bound \eqref{bound on Phi large time}, we readily have
\begin{align}\label{bound on approx nabla u and rho}
\int_1^\infty\intox|\nabla u^\eta|^4dxd\tau+\int_1^\infty\intox|\rho^\eta-\rho_s|^\frac{10}{3}dxd\tau\le CC_0^\theta.
\end{align}
Using the bounds \eqref{bounds on approx sol 1}, \eqref{bounds on approx sol 2} and \eqref{bound on approx nabla u and rho} and the estimate \eqref{bound on F and omega in terms of u} on $F$, the Sobolev inequality \eqref{GN1} gives
\begin{align*}
0\le \mathcal{F}^\eta(t)&\le C\Big(\intox|F^\eta(x,t)|^2dx\Big)^\frac{6-q_2}{4}\Big(\intox|\nabla F^\eta(x,t)|^2dx\Big)^\frac{3q_2-6}{4}\\
&\le C\Big(\intox(|\nabla u^\eta|^2+|\rho^\eta-\rho_s|^2)(x,t)dx\Big)^\frac{6-q_2}{4}\\
&\qquad\times\Big(\intox(|\dot{u}^\eta|^2+|\nabla u^\eta|^2+|\rho^\eta-\rho_s|^4)(x,t)dx\Big)^\frac{3q_2-6}{4}\\
&\le CC_0^\theta,\qquad t\ge1,
\end{align*}
and since $q_2\ge\frac{10}{3}$, 
\begin{align*}
\int_1^\infty \mathcal{F}^\eta(\tau)d\tau&\le C\int_1^\infty\Big(\intox|F^\eta(x,\tau)|^2dx\Big)^\frac{6-q_2}{4}\Big(\intox|\nabla F^\eta(x,\tau)|^2dx\Big)^\frac{3q_2-6}{4}d\tau\\
&\le C\Big(\sup_{t\ge1}\intox(|\dot{u}^\eta|^2+|\nabla u^\eta|^2+|\rho^\eta-\rho_s|^4)(x,t)dx\Big)^{\frac{3q_2-6}{4}-1}\\
&\qquad\times\Big(\sup_{t\ge1}\intox(|\nabla u^\eta|^2+|\rho^\eta-\rho_s|^2)(x,t)dx\Big)^{\frac{6-q_2}{4}}\\
&\qquad\times\Big(\int_1^\infty\intox(|\dot{u}^\eta|^2+|\nabla u^\eta|^2+|\rho^\eta-\rho_s|^4)dxd\tau\Big)\\
&\le CC_0^\theta.
\end{align*}
Recalling the fact that 
$$|F^\eta_t|=|\dot{F}^\eta-u^\eta\cdot\nabla F^\eta|\le C(|\nabla u^\eta|+|\nabla u^\eta|^2+|\nabla \dot{u}^\eta|+|u^\eta||\nabla F^\eta|),$$
we further have
\begin{align*}
&\int_1^\infty\Big|\frac{d \mathcal{F}^\eta}{dt}(\tau)\Big|d\tau\\
&\le C\int_1^\infty\intox |F^\eta|^{q_2-1}| |(\dot{F}^\eta-u^\eta\cdot\nabla F^\eta)|dxd\tau\\
&\le C\Big(\int_1^\infty\intox|F^\eta|^{2q_2-2}dxd\tau\Big)^\frac{1}{2}\Big(\int_1^\infty\intox(|\nabla u^\eta|^2+|\nabla u^\eta|^4)dxd\tau\Big)^\frac{1}{2}\\
&\qquad+C\Big(\int_1^\infty\intox |F^\eta|^{2q_2-2}dxd\tau\Big)^\frac{1}{2}\Big(\int_1^\infty\intox(|\nabla \dot{u}^\eta|^2+|u^\eta|^2|\nabla F^\eta|^2)dxd\tau\Big)^\frac{1}{2}\\
&\le CC_0^\theta.
\end{align*}
By the same method as shown above, the integral $\dis \int_1^\infty\intox|F^\eta|^{2q_2-2}dxd\tau$ is bounded by $CC_0^\theta$ for $q_2\ge\frac{8}{3}$. Moreover, using the bounds \eqref{bounds on approx sol 1} and \eqref{bound on approx nabla u and rho}, the space-time integrals on $|\nabla u^\eta|^2$, $|\nabla u^\eta|^4$ and $|\nabla\dot{u}^\eta|^2$ can be bounded by $CC_0^\theta$. To bound the integral $\dis\int_1^\infty\intox|u^\eta|^2|\nabla F^\eta|^2dxd\tau$, using \eqref{GN1} and \eqref{bound on F and omega in terms of u}, together with the bounds \eqref{bounds on approx sol 1}, \eqref{bounds on approx sol 2} and \eqref{bound on approx nabla u and rho}, we have
\begin{align*}
&\int_1^\infty\intox|u^\eta|^2|\nabla F^\eta|^2dxd\tau\\
&\le C\int_1^\infty\Big(\intox|u^\eta|^6dx\Big)^\frac{1}{3}\Big(\intox|\nabla F^\eta|^3dx\Big)^\frac{2}{3}d\tau\\
&\le C\int_1^\infty\Big(\intox|\nabla u^\eta|^2dx\Big)\Big(\intox(|\dot{u}|^3+|\nabla u^\eta|^3+|\rho^\eta-\rho_s|^6)dx\Big)^\frac{2}{3}d\tau\\
&\le C\int_1^\infty\Big(\intox|\nabla u^\eta|^2dx\Big)\Big(\intox|\dot{u}^\eta|^2dx\Big)^\frac{1}{2}\Big(\intox|\nabla\dot{u}^\eta|^2dx\Big)^\frac{1}{2}d\tau+CC_0^\theta\\
&\le CC_0^\theta.
\end{align*}
We thus obtain the following estimate on $\dis\frac{d \mathcal{F}^\eta}{dt}$:
\begin{align*}
\int_1^\infty\Big|\frac{d \mathcal{F}^\eta}{dt}(\tau)\Big|d\tau\le CC_0^\theta.
\end{align*}
Hence there is a subsequence $\eta'$ such that $\mathcal{F}^{\eta'}\to \mathcal{F}$ a.e. on $[1,\infty)$ and $\mathcal{F}$ has bounded variation. For each $T>0$, by the bounded convergence theorem, 
\begin{align*}
\int_1^T \mathcal{F}(\tau)d\tau=\lim_{\eta'\to0}\int_1^T \mathcal{F}^{\eta'}(\tau)d\tau\le CC_0^\theta,
\end{align*}
which gives $\mathcal{F}\in L^1[1,\infty)$. Therefore it implies $\dis\lim_{T\to\infty} \mathcal{F}(T)=0$. On the other hand, by the convergences \eqref{weak convergence 1} and \eqref{strong convergence on rho}, we also have 
$$\mbox{$F^{\eta'}(\cdot,t)\rightharpoonup F(\cdot,t)$ weakly in $L^2$ for all $t>0$.}$$
Since $F^{\eta'}(\cdot,t)$ are uniformly bounded in $L^{q_2}$ for all $t\ge1$, passing to a subsequence of $\eta'$ if necessary, it further implies that
$$\mbox{$F^{\eta'}(\cdot,t)\rightharpoonup F(\cdot,t)$ weakly in $L^{q_2}$ for all $t\ge1$.}$$
Therefore for almost all $t\ge1$,
\begin{align*}
\intox|F(x,t)|^{q_2}dx\le \liminf_{\eta'\to0}\intox|F^{\eta'}(x,t)|^{q_2}dx=\mathcal{F}(t).
\end{align*}
For each $\varepsilon>0$, there is a time $T>1$ such that $\mathcal{F}(t)\le\varepsilon$ when $t\ge T$, then by the H\"{o}lder continuity of $F^\eta$ in $t$ given by \eqref{bounds on approx sol 4}, for all $t\ge T$,
\begin{align*}
\intox |F(x,t)|^{q_2}dx\le C\varepsilon.
\end{align*}
Since $\varepsilon$ can be arbitrary, the above estimate shows that $F(\cdot,t)\to0$ in $L^{q_2}$ as $t\to\infty$. The case for $\omega$ follows by the same argument as for $F$. This proves \eqref{long time 2}.

The long time behaviour of $\rho$ and $u$ given by \eqref{long time 3} can be proved in the same way as given in \cite{hoff95}, so we now focus on the long time average convergence \eqref{long time 1} for $u_F$ and $F$. For $\alpha\in(\frac{1}{2},1]$, we choose $r\in(3,\infty)$ so that $$\alpha-\frac{1}{2}>1-\frac{3}{r}.$$ Define $\tilde\alpha=1-\frac{3}{r}$, then we have $\alpha-\frac{1}{2}>\tilde\alpha$. Using the Sobolev imbeddings \eqref{GN1} and \eqref{holder bound general}, together with \eqref{bound on F and omega in terms of u}-\eqref{bound on u in terms of F and omega}, we have for $\tau\in(0,T]$ that
\begin{align}\label{bound on uF and F by dot u}
\|u_F(\cdot,\tau)\|_{C^{1+\tilde\alpha}}&+\|F(\cdot,\tau)\|_{C^{\tilde\alpha}}\\
&\le C\|\dot{u}(\cdot,\tau)\|_{L^r}\notag\\
&\le C\Big(\intox|\dot{u}(x,\tau)|^2dx\Big)^\frac{6-r}{4}\Big(\intox|\nabla\dot{u}(x,\tau)|^2dx\Big)^\frac{3r-6}{4}.\notag
\end{align}
Hence for $T\le1$, we integrate \eqref{bound on uF and F by dot u} and apply \eqref{bounds on approx sol 1} to obtain
\begin{align*}
&\int_0^T(\|u_F(\cdot,\tau)\|_{C^{1+\tilde\alpha}}+\|F(\cdot,\tau)\|_{C^{\tilde\alpha}})d\tau\\
&\le C\int_0^T\Big(\intox|\dot{u}(x,\tau)|^2dx\Big)^\frac{6-r}{4}\Big(\intox|\nabla\dot{u}(x,\tau)|^2dx\Big)^\frac{3r-6}{4}d\tau\\
&\le C\int_0^T\Big(\tau^{1-\alpha}\intox|\dot{u}(x,\tau)|^2dx\Big)^\frac{6-r}{4}\Big(\tau^{2-\alpha}\intox|\nabla\dot{u}(x,\tau)|^2dx\Big)^\frac{3r-6}{4}\\
&\qquad\qquad\times \tau^{\frac{s}{2}+\frac{3}{2r}-\frac{5}{4}}d\tau\\
&\le C\Big(\int_0^T\intox \tau^{1-\alpha}|\dot{u}(x,\tau)|^2dxd\tau\Big)^\frac{6-r}{4}\Big(\int_0^T\intox \tau^{2-\alpha}|\nabla\dot{u}(x,\tau)|^2dxd\tau\Big)^\frac{3r-6}{4}\\
&\qquad\qquad\times \Big(\int_0^T\tau^{s+\frac{3}{r}-\frac{5}{2}}d\tau\Big)^\frac{1}{2}\\
&\le CC_0^\theta T^{\alpha-\tilde\alpha-\frac{1}{2}},
\end{align*}
where the last inequality follows since $s+\frac{3}{r}-\frac{5}{2}=\alpha-\tilde\alpha-\frac{3}{2}>-1$. For $T>1$, we can repeat the above computation by integrating from 1 to $T$, which gives
\begin{align*}
\int_1^T(\|u_F(\cdot,\tau)\|_{C^{1+\tilde\alpha}}+\|F(\cdot,\tau)\|_{C^{\tilde\alpha}})d\tau\le CC_0^\theta T^{\frac{1}{2}}.
\end{align*}
If we pick $\nu>\max\{\frac{1}{2},\alpha-\frac{1}{2}-\tilde\alpha\}$, then we have
\begin{align*}
\frac{1}{T^\nu}\int_0^T(\|u_F(\cdot,\tau)\|_{C^{1+\tilde\alpha}}+\|F(\cdot,\tau)\|_{C^{\tilde\alpha}})d\tau\le CC_0^\theta(T^{\frac{1}{2}-\nu}+T^{\alpha-\tilde\alpha-\frac{1}{2}-\nu}),
\end{align*}
and hence by taking $T\to\infty$, the long time average convergence \eqref{long time 1} holds. This completes the proof of Theorem~\ref{existence of weak sol thm}.

\begin{rem}\label{explanation on the time integral bound on nabla u}
Under the assumption that the initial density $\rho_0$ is piecewise H\"{o}lder continuous, by Theorem~\ref{existence of weak sol thm}, we can see that $\rho$ is piecewise H\"{o}lder continuous for positive time. As a consequence, it further implies that $\nabla u\in L^1((0,T);W^{1,\infty})$. To see how it works, we make use of the Poisson equation \eqref{decomposition of u}$_2$ again and apply properties of Newtonian potentials to conclude that the $C^{1+\beta(t)}(\R^3)$ norm of $u_P$ remains finite in finite time, hence there exists $C(T)>0$ such that the following bound holds for $\nabla u_P$ as well:
\begin{align}\label{bound on nabla uP}
\int_{0}^{T}||\nabla u_{P}(\cdot,\tau)||_{L^\infty}d\tau\le C(T).
\end{align}
Together with the bound \eqref{bound on nabla uF} on $u_F$, we conclude the estimate \eqref{bound on time integral on nabla u} for $\|\nabla u\|_{L^\infty}$, hence the regularity condition \eqref{spaces for weak sol 1} holds for the weak solution to \eqref{NS} with piecewise H\"{o}lder continuous initial density. The results of Theorem~\ref{Main thm} therefore do apply to this class of weak solutions, which includes solutions with Riemann-like initial data.
\end{rem}

\section{Stability and uniqueness of weak solutions: Proof of Theorem~\ref{Main thm}}\label{proof of main thm section}

In this section, we address the stability of weak solutions to \eqref{NS} with respect to initial data and steady state solutions and give the proof of Theorem~\ref{Main thm}. As mentioned before, weak solutions with minimal regularity are best compared in a Lagrangian framework. In other words, we aim at comparing the {\it instantaneous states} of corresponding fluid particles in two different solutions. To this end, we employ some delicate estimates on {\it particle trajectories}. More precisely, for $T>0$, the bound \eqref{bound on nabla uF} and \eqref{bound on nabla uP} guarantee the existence and uniqueness of the mapping $X(y,t,t')\in C(\R^3\times[0,T]^2)$ satisfying
\begin{align}\label{integral curve X}
\left\{ \begin{array}
{lr} \dis\frac{\partial X}{\partial t}(y,t,t')
=u(X(y,t,t'),t)\\ X(y,t',t')=y
\end{array} \right.
\end{align}
where $(\rho,u,B)$ is a weak solution to \eqref{NS}. Moreover, the mapping $X(\cdot,t,t')$ is Lipschitz on $\R^3$ for $(t,t')\in[0,T]^2$. The results are given in the following proposition and the proof can be found in \cite{hoff06}.
\begin{prop}\label{prop on X}
Let $T>0$ and $u$ satisfy \eqref{spaces for weak sol 1}. Then there is a unique function $X\in C(\R^3\times[0,T]^2$) satisfying \eqref{integral curve X}. In particular, $X(\cdot,t,t')$ is Lipschitz on $\R^3$ for $(t,t')\in[0,T]^2$, and there is a constant $C$ such that
\begin{equation*}
\Big\|\frac{\partial X}{\partial y}(\cdot,t,t')\Big\|_{L^\infty}\le C,\qquad (t,t')\in[0,T]\times[0,T].
\end{equation*}
\end{prop} 
With respect to velocities $u$ and $\bar{u}$, for $y\in\R^3$, we let $X$, $\bar{X}$ be two integral curves given by
\begin{align*}
\left\{ \begin{array}
{lr} \dis\frac{\partial X}{\partial t}(y,t,t')
=u(X(y,t,t'),t)\\ X(y,t',t')=y
\end{array} \right.
\end{align*}
and
\begin{align*}
\left\{ \begin{array}
{lr} \dis\frac{\partial \bar{X}}{\partial t}(y,t,t')
=\bar{u}(\bar{X}(y,t,t'),t)\\ \bar{X}(y,t',t')=y.
\end{array} \right.
\end{align*}
We then define $S(x,t)$, $S^{-1}(x,t)$ by
\begin{equation}\label{def of S}
S(x,t)=\bar{X}(X(x,0,t),t,0),
\end{equation}
and 
\begin{equation}\label{def of S-1}
S^{-1}(x,t)=X(\bar{X}(x,0,t),t,0).
\end{equation}
The following proposition provides some properties of $S$ and $S^{-1}$, which will be crucial for later analysis.
\begin{prop}\label{prop on S}
Let $S$ and $S^{-1}$ be as given in \eqref{def of S}-\eqref{def of S-1}. Then we have:

\noindent{\bf (a)} $S^{\pm1}$ is continuous on $R^3\times[0,T]$ and Lipschitz continuous on $R^3\times[\tau,T]$ for all $\tau>0$, and there is a constant C such that $$\|\nabla S^{\pm1}(\cdot,\tau)\|_{L^\infty}\le C,\qquad t\in[0,T];$$

\noindent{\bf (b)} $(S_t+\nabla S u)(x,t)=\bar{u}(S(x,t),t)$ a.e. in $\R^3\times(0,T);$

\noindent{\bf (c)} $\bar{\rho}(S(x,t),t)\rho_0(X(x,0,t))\det\nabla S(x,t)=\rho(x,t)\bar{\rho}_0(X(x,0,t))$ a.e. in $\R^3\times(0,T)$;

\noindent{\bf (d)} If $u, u \in L^2(\R^3 \times (0,T))$, then for all $t\in(0,T)$,
\begin{align}
\intox|x-S(x,t)|^2dx&\le Ct\int_0^t\intox|u(x,\tau)-\bar{u}(S(x,\tau),\tau)|^2dxd\tau,\label{bound on S}\\
\intox|x-S^{-1}(x,t)|^2dx&\le Ct\int_0^t\intox|u(S^{-1}(x,\tau),\tau)-\bar{u}(x,\tau)|^2dxd\tau.\label{bound on S -1}
\end{align}
\end{prop}
\begin{proof}
The proof can be found in \cite[pp. 1752]{hoff06}.
\end{proof}

We are now ready to give the proof of Theorem~\ref{Main thm}. Throughout this section, $C$ always denotes a generic positive constant which depends on the parameters $P$, $f$, $\lambda$, $\mu$, $T$, $a$, $r$ as described in Theorem~\ref{Main thm}. First, in view of the weak form \eqref{WF2} of the momentum equation, we let $\psi:\R^3\times[0,T]\rightarrow\R^3$ be a test function satisfying
\begin{align}\label{weak form for rho}
&-\intox \rho_0(x)u_0(x)\psi(x,0)dx\notag\\
&=\int_0^T\intox \Big[\rho u\cdot(\psi_t+\nabla\psi u)+(P(\rho)-P_s)\divv(\psi)-\mu\nabla u^j\cdot\nabla\psi^j\\
&\qquad\qquad\qquad-\lambda\divv(u)\divv(\psi)-\lambda(\divv(u)\divv(\psi)+(\rho - \rho_s)f\cdot\psi\Big]dxd\tau,\notag
\end{align}
where we used the expressions \eqref{condition on f} and \eqref{eqn for steady state}$_1$ for $\nabla P_s$. Define $\bar\psi=\psi\circ S^{-1}$. Then we have
\begin{align}\label{weak form for bar rho}
&-\intox \bar{\rho}_0(x)\bar{u}_0(x)\bar{\psi}(x,0)dx\notag\\
&=\int_0^T\intox \Big[\bar{\rho}\bar{u}\cdot(\bar{\psi}_t+\nabla\bar{\psi}\bar{u})+(P(\bar{\rho})-\bar{P_s})\divv(\bar{\psi})-\mu\nabla\bar{u}^j\cdot\nabla\bar{\psi}^j\\
&\qquad\qquad\qquad-\lambda\divv(\bar{u})\divv(\bar{\psi})-\lambda(\divv(\bar{u})\divv(\bar{\psi})+(\bar{\rho} - \bar{\rho}_s)\bar{f}\cdot\bar{\psi}\Big]dxd\tau\notag,
\end{align}
with $\bar{P_s}=P(\bar{\rho}_s)$. For the term involving $\bar{\psi}_t$, we can rewrite it as
\begin{align*}
&\intoxt \bar{\rho}\bar{u}\cdot(\bar{\psi}_t+\nabla\bar{\psi}\bar{u})dxd\tau\\
&=\intoxt \bar{\rho}(S)\bar{u}(S)\cdot(\bar{\psi}_t(S)+\nabla\bar{\psi}\bar{u}(S))|\det(\nabla S)|dxd\tau\\
&=\intoxt A_0\rho\bar{u}(S)(\psi_t+\nabla\psi u)dxd\tau,
\end{align*}
where $A_0$ is given by
\begin{align*}
A_0(x,t)=\frac{\bar{\rho}_0(X(x,0,t))}{\rho_0(X(x,0,t))}
\end{align*}
and we used the fact that $A_0\rho=(\bar{\rho}\circ S)|\det(\nabla S)|$ from Proposition~\ref{prop on S}. And by using the definition of $\bar{F}$ and $\omega$ from \eqref{def of F and omega} (replacing $F$ by $\bar{F}$, $u$ by $\bar{u}$, etc.), we can further rewrite the terms on the right side of \eqref{weak form for bar rho} as follows.
\begin{align*}
&\intoxt\left[(\bar{P_s}-\bar{P})\divv(\bar\psi)+\mu\nabla\bar{u}^j\cdot\nabla\bar{\psi}^j+\lambda\divv(\bar{u})\divv(\bar{\psi})\right]dxd\tau\\
&=\intoxt\left[(\mu+\lambda)\divv(\bar{u})-\bar{P}+\bar{P_s}\right]\divv(\bar{\psi})dxd\tau\\
&\qquad+\intoxt\mu(\bar{u}^j_{x_k}-\bar{u}^k_{x_j})\bar{\psi}^j_{x_k}dxd\tau\\
&=-\intoxt[\nabla(\bar{\rho}_s\bar{F})\cdot\psi+\mu\bar{\omega}^{j,k}_{x_k}\psi^j]dxd\tau\\
&\qquad+\int_0^T\intox \left[\nabla(\bar{\rho}_s\bar{F})\cdot(\psi-\psi\circ S^{-1})+\mu\bar{\omega}^{j,k}_{x_k}(\psi^j-\psi^j\circ S^{-1})\right]dxd\tau\\
&=\int_0^T\intox \left[(P_s-P)\divv(\psi)+\mu\nabla\bar{u}^j\cdot\nabla\psi^j+\lambda\divv(\bar{u})\divv(\psi)\right]dxd\tau\\
&\qquad+\int_0^T\intox \left[\nabla(\bar{\rho}_s\bar{F})\cdot(\psi-\psi\circ S^{-1})+\mu\bar{\omega}^{j,k}_{x_k}(\psi^j-\psi^j\circ S^{-1})\right]dxd\tau.
\end{align*}
Hence by taking the difference between \eqref{weak form for rho} and \eqref{weak form for bar rho}, for all $\psi$ and $\bar{\psi}$, we have the following expression on $\dis\intox (\bar{\rho}\bar{u}_0-\rho_0u_0)\cdot\psi(x,0)dx$:
\begin{align}\label{estimate on weak form for u and rho}
&\intox (\bar{\rho}\bar{u}_0-\rho_0u_0)\cdot\psi(x,0)dx\notag\\
&=\int_0^T\intox \left[\rho(u-\bar{u}\circ S)(\psi_t+\nabla\psi u)+(1-A_0)\rho(\bar{u}\circ S)(\psi_t+\nabla\psi u)\right]dxd\tau\\
&\qquad+\int_0^T\intox \left[(P_s-P)\divv(\psi)+\mu\nabla\bar{u}^j\cdot\nabla\psi^j+\lambda\divv(\bar{u})\divv(\psi)\right]dxd\tau\notag\\
&\qquad+\int_0^T\intox \left[(1-A_0)\rho(\bar{u}\circ S)(\psi_t+\nabla\psi u)\right]dxd\tau\notag\\
&\qquad+\int_0^T\intox \left[\nabla(\bar{\rho}_s\bar{F})\cdot(\psi-\psi\circ S^{-1})+\mu\bar{\omega}^{j,k}_{x_k}(\psi^j-\psi^j\circ S^{-1})\right]dxd\tau\notag\\
&\qquad+\int_0^T\intox (\bar{u}\circ S-\bar{u})(\mu\Delta\psi+\lambda\nabla\divv(\psi))dxd\tau\notag\\
&\qquad+\int_0^T\intox (P-\bar{P})\divv(\psi)+\int_0^T\intox (\bar{P_s}-P_s)\divv(\psi)dxd\tau\notag\\
&\qquad+\int_0^T\intox \left(\bar{\rho}_s(\bar{f}\circ S)\cdot\psi-\rho_s f\cdot\psi\right)dxd\tau\notag.
\end{align}
Next we extend $\rho$, $u$ to be constant in time $t$ outside $[0,T]$ and let $\rho^\varepsilon$ and $u^\varepsilon$ be the corresponding smooth approximation obtained by mollifying in both $x$ and $t$. Then we define $\psi^\varepsilon:\R^3\times[0,T]\to\R^3$ to be the solutions satisfying
\begin{align*}
\left\{ \begin{array}
{lr} \rho^\varepsilon(\psi^\varepsilon_t+u^\varepsilon\cdot\nabla\psi^\varepsilon)+\mu\Delta\psi^\varepsilon+\lambda\nabla\divv(\psi^\varepsilon)=G\\ 
\psi^\varepsilon(\cdot,\tau)=0.
\end{array} \right.
\end{align*}
By simple elliptic estimates (or refer to \cite[Lemma~3.1]{hoff06}), it can be shown that $\psi^\varepsilon$ satisfies the following bounds in terms of $G$:
\begin{align}\label{bound on psi 1}
&\sup_{0\le \tau\le T}\intox[|\psi^\varepsilon(x,\tau)|^2+|\nabla\psi^\varepsilon(x,\tau)|^2]dx\\
&\qquad+\int_0^T\intox[|\psi^\varepsilon_t+\nabla\psi^\varepsilon u^\varepsilon|^2+|D^2_x\psi^\varepsilon|^2]dxd\tau\le C\int_0^T\intox|G|^2dxd\tau\notag,
\end{align}
\begin{align}\label{bound on psi 2}
\sup_{0\le \tau\le T}\|\psi^\varepsilon(\cdot,\tau)\|_{L^\infty}+\int_0^T\intox|\psi^\varepsilon|^qdxd\tau\le C(G),
\end{align}
where $q>6$ is given by \eqref{condition on q and viscosity} and $C(G)$ is a positive constant which depends on $G$. We now take $\psi=\psi^\varepsilon$ in \eqref{estimate on weak form for u and rho} to obtain
\begin{align}\label{estimate on weak form for u and rho 2}
\intox (\bar{\rho}\bar{u}_0-\rho_0u_0)\cdot\psi^\varepsilon(x,0)dx=\intoxt z\cdot Gdxd\tau+\sum_{i=1}^7\mathcal{R}_i,
\end{align}
where $z=u-\bar{u}\circ S$ and $\mathcal{R}_1,\dots,\mathcal{R}_7$ are given by:
\begin{align*}\
&\mathcal{R}_1=\intoxt\left[\nabla(\bar{\rho}_s\bar{F})\cdot(\psi^\varepsilon-\psi^\varepsilon\circ S^{-1})+\mu\bar{\omega}^{j,k}_{x_k}(\psi^\varepsilon-\psi^\varepsilon\circ S^{-1})\right]dxd\tau,\\
&\mathcal{R}_2=\intoxt \left[\rho(f-\bar{f}\circ D)\cdot\psi^\varepsilon+(1-A_0)\rho(\bar{f}\circ S)\cdot\psi^\varepsilon\right]dxd\tau,\\
&\mathcal{R}_3=\intoxt(\bar{u}\circ S-\bar{u})\cdot(\mu\Delta\psi^\varepsilon+\lambda\divv(\psi^\varepsilon))dxd\tau,\\
&\mathcal{R}_4=\intoxt z\cdot \left[(\rho-\rho^\varepsilon)\psi^\varepsilon_t+\nabla\psi^\varepsilon(\rho u-\rho^\varepsilon u^\varepsilon)\right]dxd\tau,\\
&\mathcal{R}_5=\intoxt(1-A_0)\rho(\bar{u}\circ S)\cdot(\psi^\varepsilon_t+\nabla\psi^\varepsilon u)dxd\tau,\\
&\mathcal{R}_6=\intoxt(P-\bar{P})\divv(\psi^\varepsilon)+\int_0^T\intox (\bar{P_s}-P_s)\divv(\psi^\varepsilon)dxd\tau,\notag\\
&\mathcal{R}_7=\int_0^T\intox \left(\bar{\rho}_s(\bar{f}\circ S)\cdot\psi^\varepsilon-\rho_s f\cdot\psi^\varepsilon\right)dxd\tau\notag.
\end{align*}
The left side of \eqref{estimate on weak form for u and rho 2} can be readily bounded by 
\begin{align}\label{estimate of LHS}
\Big|\intox (\bar{\rho}\bar{u}_0-\rho_0u_0)\cdot\psi^\varepsilon(x,0)dx\Big|\le\|\rho_0u_0-\bar{\rho}_0\bar{u}_0\|_{L^2}\left(\intoxt|G|^2dxd\tau\right)^\frac{1}{2}.
\end{align}
We now aim at controlling the terms on the right side of \eqref{estimate on weak form for u and rho 2}. Following the method given in \cite{hoff06} and with the help of the bounds \eqref{bound on psi 1}-\eqref{bound on psi 2}, the terms $\mathcal{R}_2$, $\mathcal{R}_3$, $\mathcal{R}_4$ and $\mathcal{R}_5$ satisfy the following estimates:
\begin{align}\label{estimate on R2}
|\mathcal{R}_2|&\le C\left[\left(\intox|f-\bar{f}\circ S|^2dx\right)^\frac{1}{2}+\|\bar{f}\|_{L^{2\tilde q}}\|\rho_0-\bar{\rho}_0\|_{L^2\cap L^{2\tilde q'}}\right]\\
&\qquad\qquad\times\left(\intoxt|G|^2dxd\tau\right)^\frac{1}{2},\notag
\end{align}
\begin{align}\label{estimate on R3}
|\mathcal{R}_3|\le C\left(\intoxt|z|^2dxd\tau\right)^\frac{1}{2}\left(\intoxt|G|^2dxd\tau\right)^\frac{1}{2},
\end{align}
\begin{align}\label{estimate on R4}
\lim_{\varepsilon\to0}\mathcal{R}_4=\lim_{\varepsilon\to0}\intoxt z\cdot \left[(\rho-\rho^\varepsilon)\psi^\varepsilon_t+\nabla\psi^\varepsilon(\rho u-\rho^\varepsilon u^\varepsilon)\right]dxd\tau=0,
\end{align}
and
\begin{align}\label{estimate on R5}
|\mathcal{R}_5|\le C\left[\|\rho_0-\bar{\rho}_0\|_{L^2}+\left(\intoxt|z|^2dxd\tau\right)^\frac{1}{2}\right]\left(\intoxt|G|^2dxd\tau\right)^\frac{1}{2},
\end{align}
where we recall that $\tilde q$ and $\tilde q'$ are given in \eqref{condition on weak sol9}-\eqref{condition on weak sol10}. It remains to estimate the terms $\mathcal{R}_1$, $\mathcal{R}_6$ and $\mathcal{R}_7$. For $\mathcal{R}_1$, using H\"{o}lder inequality and the bound \eqref{bound on psi 1} on $\dis\intoxt|D^2_x\psi^\varepsilon|^2dxd\tau$, we have
\begin{align*}
|\mathcal{R}_1|&\le C\left(\intoxt|z|^2dxd\tau\right)^\frac{1}{2}\int_0^T \tau^\frac{1}{2}\|\nabla(\bar{\rho}_s\bar{F})(\cdot,\tau)\|_{L^4}\|\nabla\psi^\varepsilon(\cdot,\tau)\|_{L^4}d\tau\\
&\qquad+C\left(\intoxt|z|^2dxd\tau\right)^\frac{1}{2}\int_0^T \tau^\frac{1}{2}\|\nabla\bar{\omega}(\cdot,\tau)\|_{L^4}\|\nabla\psi^\varepsilon(\cdot,\tau)\|_{L^4}d\tau\\
&\le C\left(\intoxt|z|^2dxd\tau\right)^\frac{1}{2}\left(\intoxt|G|^2dxd\tau\right)^\frac{1}{8}\left(\intoxt|D^2_x\psi^\varepsilon|^2dxd\tau\right)^\frac{3}{8}\\
&\qquad\times\left(\int_0^T \tau^\frac{4}{5}(\|\nabla(\bar{\rho}_s\bar{F})(\cdot,\tau)\|^\frac{8}{5}_{L^4}+\|\nabla\bar{\omega}(\cdot,\tau)\|^\frac{8}{5}_{L^4})d\tau\right)^\frac{5}{8}\\
&\le C\left(\intoxt|z|^2dxd\tau\right)^\frac{1}{2}\left(\intoxt|G|^2dxd\tau\right)^\frac{1}{2}\\
&\qquad\times\left(\int_0^T \tau^\frac{4}{5}(\|\nabla(\bar{\rho}_s\bar{F})(\cdot,\tau)\|^\frac{8}{5}_{L^4}+\|\nabla\bar{\omega}(\cdot,\tau)\|^\frac{8}{5}_{L^4})d\tau\right)^\frac{5}{8}.
\end{align*}
Using \eqref{bound on F and omega in terms of u} and the boundedness assumption \eqref{condition on weak sol6} on $\bar{\rho}_s$, the term involving $\bar{F}$ and $\bar{\omega}$ can be bounded by $\dis\left(\int_0^T \tau^\frac{4}{5}\left(\intox|\dot{\bar{u}}(x,\tau)|^4dx\right)^\frac{3}{5}d\tau\right)^\frac{5}{8}$, and with the help of \eqref{GN1} and the energy estimates \eqref{energy bound on weak sol}, we further have
\begin{align*}
&\int_0^T \tau^\frac{4}{5}\left(\intox|\dot{\bar{u}}|^4dx\right)^\frac{3}{5}d\tau\\
&\le C\int_0^T \tau^\frac{4}{5}\left(\intox|\dot{\bar{u}}|^2dx\right)^\frac{1}{5}\left(\intox|\nabla\dot{\bar{u}}|^2dx\right)^\frac{3}{5}d\tau\\
&\le C\left(\int_0^T \tau^{4\alpha-3}d\tau\right)^\frac{1}{5}\left(\int_0^T \tau^{1-\alpha}\intox|\dot{\bar{u}}|^2dxd\tau\right)^\frac{1}{5}\left(\int_0^T \tau^{2-\alpha}\intox|\nabla\dot{\bar{u}}|^2dxd\tau\right)^\frac{1}{5}\\
&\le CT^\frac{4\alpha-2}{5},
\end{align*}
and we recall that $\alpha\in(\frac{1}{2},1]$. Hence we conclude
\begin{align}\label{rough bound on R1}
|\mathcal{R}_1|\le CT^\frac{2\alpha-1}{4}\left(\intoxt|z|^2dxd\tau\right)^\frac{1}{2}\left(\intoxt|G|^2dxd\tau\right)^\frac{1}{2}.
\end{align}
In particular, for $[t_1,t_2]\subseteq[0,T]$, if we define
\begin{equation*}
\mathcal{R}_1(t_1,t_2)=\int_{t_1}^{t_2}\intox\left[\nabla(\bar{\rho}_s\bar{F})\cdot(\psi^\varepsilon-\psi^\varepsilon\circ S^{-1})+\mu\bar{\omega}^{j,k}_{x_k}(\psi^\varepsilon-\psi^\varepsilon\circ S^{-1})\right]dxd\tau,
\end{equation*}
then we also have
\begin{equation}\label{estimate on R1}
|\mathcal{R}_1(t_1,t_2)|\le C|t_2-t_1|^\frac{2\alpha-1}{4}\left(\intoxts|z|^2dxd\tau\Big)^\frac{1}{2}\Big(\intoxts|G|^2dxd\tau\right)^\frac{1}{2}.
\end{equation}

For $\mathcal{R}_6$, we define the function $\chi(x,t)$ by
\begin{align*}
\chi(x,t)=\frac{P(\rho(x,t))-P(\bar{\rho}(x,t))}{\rho(x,t)-\bar{\rho}(x,t)},
\end{align*}
then by the mean value theorem, we can further rewrite $\chi(x,t)$ as
\begin{align*}
\chi(x,t)=\int_0^1 P'(\rho(x,t)+\theta(\bar{\rho}(x,t)-\rho(x,t)))d\theta.
\end{align*}
For $\psi\in\mathcal{D}(\R^3)$ and $t\ge0$, 
\begin{align*}
\Big|\intox[P(\rho)-P(\bar{\rho}]\psi (x,\tau)dx\Big|&=\Big|\intox\chi(\rho-\bar{\rho})\psi (x,\tau)dx\Big|\\
&\le C\|(\rho-\bar{\rho})(\cdot,\tau)\|_{H^{-1}}(\|\psi\|_{H^1}+\|\psi\nabla\chi\|_{L^2}).
\end{align*}
If $P(\rho)=a\rho$, then it is clear that $\chi$ is a constant and $\nabla\chi=0$. For more general pressure $P$ satisfying \eqref{more general condition on P}, notice that for $r\in[3,\infty]$, if we take $p=2+\frac{4}{r-2}$, then $p\in[2,6]$ with $r=\frac{2p}{p-2}$. Hence by the Sobolev imbedding \eqref{GN1} and the boundedness condition \eqref{more general condition on P} on $\nabla\chi$,
\begin{align*}
\|\psi\nabla\chi\|_{L^2}&\le \|\psi\|_{L^p}\Big(\intox|\nabla\chi|^\frac{2p}{p-2}dx\Big)^\frac{p-2}{2p}\\
&\le \|\psi\|_{H^1}\|\nabla\chi\|_{L^r}\le C\|\psi\|_{H^1}.
\end{align*}
Therefore in either cases, we have
\begin{align}\label{H -1 bound on P}
\|(P-\bar{P})(\cdot,\tau)\|_{H^{-1}}\le C\|(\rho-\bar{\rho})(\cdot,\tau)\|_{H^{-1}}.
\end{align}
Hence together with the bound \eqref{bound on psi 2} on $\psi^\varepsilon$, it implies that
\begin{align}\label{bound on R6 1}
&\Big|\intoxt(P-\bar{P})\divv(\psi^\varepsilon)dxd\tau\Big|\\
&\qquad\le\int_0^T\|(P-\bar{P})(\cdot,\tau)\|_{H^{-1}}\|\divv(\psi^\varepsilon)\|_{H^1}d\tau\notag\\
&\qquad\le C\sup_{0\le \tau\le T}\|(\rho-\bar{\rho})(\cdot,\tau)\|_{H^{-1}}\left(\intoxt|G|^2dxd\tau\right)^\frac{1}{2}.\notag
\end{align}
Following the argument given in \cite{hoff06}, the term $\dis\sup_{0\le \tau\le T}\|(\rho-\bar{\rho})(\cdot,\tau)\|_{H^{-1}}$ can be bounded by
\begin{align*}
\sup_{0\le \tau\le T}\|(\rho-\bar{\rho})(\cdot,\tau)\|_{H^{-1}}\le C\left[\|\rho_0-\bar{\rho}_0\|_{L^2}+T^\frac{1}{2}\left(\intoxt|z|^2dxd\tau\right)^\frac{1}{2}\right],
\end{align*}
and we conclude from \eqref{bound on R6 1} that
\begin{align}\label{bound on R6 2}
&\Big|\intoxt(P-\bar{P})\divv(\psi^\varepsilon)dxd\tau\Big|\\
&\le C\left[\|\rho_0-\bar{\rho}_0\|_{L^2}+\left(\intoxt|z|^2dxd\tau\right)^\frac{1}{2}\right]\left(\intoxt|G|^2dxd\tau\right)^\frac{1}{2}.\notag
\end{align}
On the other hand, the term $\dis\int_0^T\intox (\bar{P_s}-P_s)\divv(\psi^\varepsilon)dxd\tau$ can be readily bounded by
\begin{align}\label{bound on R6 3}
\Big|\int_0^T\intox (\bar{P_s}-P_s)\divv(\psi^\varepsilon)dxd\tau\Big|\le C\left(\intox|\rho_s-\bar{\rho}_s|^2dxd\tau\right)^\frac{1}{2},
\end{align}
and therefore the bounds \eqref{bound on R6 2}-\eqref{bound on R6 3} together imply
\begin{align}\label{estimate on R6}
|\mathcal{R}_6|&\le C\left[\|\rho_0-\bar{\rho}_0\|_{L^2}+\left(\intoxt|z|^2dxd\tau\right)^\frac{1}{2}\right]\left(\intoxt|G|^2dxd\tau\right)^\frac{1}{2}\\
&\qquad+\left(\intox|\rho_s-\bar{\rho}_s|^2dxd\tau\right)^\frac{1}{2}.\notag
\end{align}
Finally for the term $\mathcal{R}_7$, we can rewrite it as follows.
\begin{align*}
\mathcal{R}_7=\intoxt(\bar{\rho}_s-\rho_s)f\cdot\psi^\varepsilon dxd\tau+\intoxt\bar{\rho}_s(\bar{f}\circ S-f)\cdot\psi^\varepsilon dxd\tau.
\end{align*}
Using \eqref{bound on psi 1}, the term $\dis\intoxt(\bar{\rho}_s-\rho_s)f\cdot\psi^\varepsilon dxd\tau$ can be bounded by
\begin{align*}
&\Big|\intoxt(\bar{\rho}_s-\rho_s)f\cdot\psi^\varepsilon dxd\tau\Big|\\
&\qquad\le C\|f\|_{L^\infty}\left(\intox|\rho_s-\bar{\rho}_s|^2dx\right)^\frac{1}{2}\left(\intoxt|\psi^\varepsilon|^2dxd\tau\right)^\frac{1}{2}\\
&\qquad\le C\|f\|_{L^\infty}\left(\intox|\rho_s-\bar{\rho}_s|^2dx\right)^\frac{1}{2}\left(\intoxt|G|^2dxd\tau\right)^\frac{1}{2},
\end{align*}
and similarly, $\dis\intoxt\bar{\rho}_s(\bar{f}\circ S-f)\cdot\psi^\varepsilon dxd\tau$ can be bounded by
\begin{align*}
&\intoxt\bar{\rho}_s(\bar{f}\circ S-f)\cdot\psi^\varepsilon dxd\tau\\
&\le C\|\bar{\rho}_s\|_{L^\infty}\left(\int_0^T\intox|f-\bar{f}\circ S|^2dxd\tau\right)^\frac{1}{2}\left(\intoxt|G|^2dxd\tau\right)^\frac{1}{2}.
\end{align*}
Recalling the assumptions \eqref{condition on weak sol3} and \eqref{condition on weak sol6}, we therefore obtain
\begin{align}\label{estimate on R7}
|\mathcal{R}_7|&\le C\left(\intox|\rho_s-\bar{\rho}_s|^2dx\right)^\frac{1}{2}\left(\intoxt|G|^2dxd\tau\right)^\frac{1}{2}\\
&\qquad+C\left(\int_0^T\intox|f-\bar{f}\circ S|^2dxd\tau\right)^\frac{1}{2}\left(\intoxt|G|^2dxd\tau\right)^\frac{1}{2}.\notag
\end{align}
Combining the estimates \eqref{estimate of LHS}, \eqref{estimate on R2}, \eqref{estimate on R3}, \eqref{estimate on R4}, \eqref{estimate on R5}, \eqref{estimate on R1}, \eqref{estimate on R6} and \eqref{estimate on R7}, we arrive at
\begin{align}\label{estimate on z in T}
\Big|\intoxt z\cdot Gdxd\tau\Big|\le C\Big[M_0\Big(\intoxt|G|^2dxd\tau\Big)^\frac{1}{2}+|\mathcal{R}_1(0,T)|\Big],
\end{align}
where $M_0$ is given by
\begin{align*}
M_0&=\|\rho_0-\bar{\rho}_0\|_{L^2\cap L^{2\tilde q'}}+\|\rho_0u_0-\bar{\rho}_0\bar{u}_0\|_{L^2}+T^{\delta}\Big(\intoxt|z|^2dxd\tau\Big)^\frac{1}{2}\\
&\qquad+\left(\intox|\rho_s-\bar{\rho}_s|^2dx\right)^\frac{1}{2}+\left(\int_0^T\intox|f-\bar{f}\circ S|^2dxd\tau\right)^\frac{1}{2}
\end{align*}
for some $\delta>0$, and $C>0$ is now fixed. Following the analysis given in \cite[pp. 1758--1759]{hoff06}, there exists a small time $\tilde\tau>0$ such that 
\begin{equation*}
\Big(\int_0^{\tilde\tau}\intox|z|^2dxd\tau\Big)^\frac{1}{2}\le 2CM_0,
\end{equation*}
and consequently
\begin{equation*}
|\mathcal{R}_1(0,\tilde\tau)|\le M_0\Big(\int_0^{\tilde\tau}\intox|G|^2dxd\tau\Big)^\frac{1}{2}.
\end{equation*}
By applying \eqref{estimate on z in T} with $T$ replaced by $2\tilde\tau$, we get
\begin{equation*}
\Big(\int_0^{2\tilde\tau}\intox|z|^2dxd\tau\Big)^\frac{1}{2}\le 4CM_0.
\end{equation*}
Since $\tilde\tau>0$ is fixed, we can exhaust the interval $[0,T]$ in finitely many steps to obtain that
\begin{equation*}
\Big(\intoxt|z|^2dxd\tau\Big)^\frac{1}{2}\le CM_0,
\end{equation*}
for some new constant $C>0$. Hence the term $\dis T^{\delta}\Big(\intoxt|z|^2dxd\tau\Big)^\frac{1}{2}$ can be eliminated from the definition of $M_0$ by a Gronw\"{a}ll-type argument. Therefore we conclude that
\begin{align}
\Big|\intoxt z\cdot Gdxd\tau\Big|&\le CM_0\Big(\intoxt|G|^2dxd\tau\Big)^\frac{1}{2}.\label{L2 bound on u uniqueness}
\end{align}
Since the bound \eqref{L2 bound on u uniqueness} holds for any $G\in H^\infty(\R^3\times[0,T])$, it shows that the term $\|z\|_{L^2([0,T]\times\R^3)}$ can be bounded by $M_0$. Finally, using the bound \eqref{spaces for weak sol 1} on the time integral on $\|\nabla\bar{u}\|_{L^\infty}$, we have
\begin{align}\label{bound on bar u and bar u S}
\intoxt|\bar{u}-\bar{u}\circ S|^2dxd\tau&\le\int_0^T\|\nabla\bar{u}(\cdot,\tau)\|^2_{L^\infty}\intox|x-S(x,\tau)|^2dxd\tau\\
&\le C\intoxt|z|^2dxd\tau,\notag
\end{align}
and hence \eqref{bound on difference} follows. This completes the proof of Theorem~\ref{Main thm}.
\begin{rem}\label{explanation on f}
If we further assume that $\|\nabla\bar{f}\|_{L^\infty}<\infty$, then it gives
\begin{align*}
&\left(\int_0^T\intox|f-\bar{f}\circ S|^2dxd\tau\right)^\frac{1}{2}\\
&\le C\left(\int_0^T\intox|\bar{f}-\bar{f}\circ S|^2dxd\tau\right)^\frac{1}{2}+CT^\frac{1}{2}\left(\intox|f-\bar{f}|^2dx\right)^\frac{1}{2}\\
&\le C\|\nabla\bar{f}\|_{L^\infty}\left(\int_0^T\intox|x-S|^2dxd\tau\right)^\frac{1}{2}+CT^\frac{1}{2}\left(\intox|f-\bar{f}|^2dx\right)^\frac{1}{2}.
\end{align*}
Using \eqref{bound on S} and \eqref{bound on bar u and bar u S}, we have
\begin{align*}
C\|\nabla\bar{f}\|_{L^\infty}\left(\int_0^T\intox|x-S|^2dxd\tau\right)^\frac{1}{2}\le C\|\nabla\bar{f}\|_{L^\infty}\Big(\intoxt|z|^2dxd\tau\Big)^\frac{1}{2},
\end{align*}
hence we can replace the difference $f-\bar{f}\circ S$ by $f-\bar{f}$ in \eqref{bound on difference} provided that $\|\nabla\bar{f}\|_{L^\infty}<\infty$.
\end{rem}



\bibliographystyle{amsalpha}

\end{document}